\documentclass[12pt]{amsart}

\usepackage{amssymb,bm,mathrsfs}

\usepackage{enumerate}

\usepackage[centering]{geometry}
\usepackage[pagebackref]{hyperref}

\theoremstyle{plain}
\newtheorem{thm}{Theorem}[section]

\newtheorem{lem}{Lemma}[section]

\theoremstyle{definition}
\newtheorem{defn}{Definition}[section]
\newtheorem{exmp}{Example}[section]

\theoremstyle{remark}

\numberwithin{equation}{section}

\DeclareMathOperator{\hdim}{dim_H}
\DeclareMathOperator{\pdim}{dim_P}
\DeclareMathOperator{\xct}{Exact}
\DeclareMathOperator{\dist}{dist}
\DeclareMathOperator{\HCF}{HCF}
\DeclareMathOperator{\Bad}{Bad}

\newcommand{\C}{\mathbb C}

\newcommand{\Q}{\mathbb Q}
\newcommand{\R}{\mathbb R}
\newcommand{\Z}{\mathbb Z}
\newcommand{\dmn}{\mathfrak F}
\newcommand{\cld}{\mathcal C}
\newcommand{\cst}{\mathfrak c}

\newcommand{\lm}{\mathcal L}

\begin{document}

\title{Sets of Exact Approximation Order by Complex rational numbers}

\author{YuBin He}

\address{Department of Mathematics, South China University of Technology,	Guangzhou, 510641, P.~R.\ China}

\email{yubinhe007@gmail.com}

\author[Ying Xiong]{Ying Xiong$^*$}

\address{Department of Mathematics, South China University of Technology, Guangzhou, 510641, P.~R.\ China}

\email{xiongyng@gmail.com}



\thanks{$^*$Corresponding author.}

\thanks{This work is supported by National Natural Science Foundation of China (Grant No.~11871227, 11771153, 11471124), Guangdong Natural Science Foundation (Grant No.~2021A1515010056, 2018B0303110005), the Fundamental Research Funds for the Central Universities, SCUT (Grant No.~2020ZYGXZR041).}

\begin{abstract}
	Given a non-increasing function $\psi$, let $\xct(\psi)$ be the set of complex numbers which are approximable by complex rational numbers to order $\psi$ but to no better order. In this paper, we obtain the Hausdorff dimension and packing dimension of $\xct(\psi)$ when $\psi(x)=o(x^{-2})$. Moreover, without the condition $\psi(x)=o(x^{-2})$, we also prove that the Hausdorff dimension of $\xct(\psi)$ is greater than $2-\tau/(1-2\tau)$ when $0<\tau=\limsup_{x\to+\infty}x^2\psi(x)$ small enough.
\end{abstract}

\maketitle

\section{Introduction}\label{sc:intro}

\subsection{Background}\label{ss:bg}

In the theory of Diophantine approximation, it is an interesting problem to consider the irrational numbers which are approximable by rational numbers to order~$\psi$ exactly. Given a function~$\psi\colon(0,\infty)\to(0,\infty)$, let
\[ W_\R(\psi):=\bigl \{x\in\R\colon |x-{p}/{q}|<\psi(|q|)\text{ for infinitely many } {p}/{q}\in \Q\bigr\} \]
be the set of real numbers which are $\psi$-well approximable. Let
\[ \xct_\R(\psi):=W_\R(\psi)\setminus\bigcup_{k\ge2}W_\R\bigl( (1-1/k)\psi \bigr) \]
be the set of real numbers approximable to order~$\psi$ but to no better order.

The Hausdorff dimension of $\psi$-well approximable sets was well studied for the cases of real numbers and linear forms, see for instance~\cite{Besic34,Bugea03,Dicki97,Dodso91,Dodso92,Jarn29,Jarn31,KimLi19,WanWu17}.

As for the set $\xct_\R(\psi)$, Jarn\'ik~\cite{Jarn31} used the theory of continued fractions to show that it is uncountable if $\psi$ is non-increasing and satisfies $\psi(x)=o(x^{-2})$. However, Jarn\'ik's method provides no information regarding Hausdorff dimension.

In 2001, Beresnevich, Dickinson and Velani~\cite{BeDiV01} obtained a precise metric statement for sets of $\psi$-well
approximable systems of linear forms which extend Jarn\'ik’s result. They further asked a question at the end of~\cite{BeDiV01} that how to determine the Hausdorff dimension of the set $\Bad_\R(\psi)$ consisting of real numbers $x$ such that $|x-p/q|\le\psi(q)$ for infinitely many $p/q\in\Q$, but $|x-p/q|\ge c(x)\psi(q)$ for all $p/q\in\Q$, where $c(x)$ is a positive constant dependent only on~$x$. Clearly, $\xct_\R(\psi)\subset\Bad_\R(\psi)\subset W_\R(\psi)$.

This problem was solved by Bugeaud~\cite{Bugea03,Bugea08}, Bugeaud and Moreira~\cite{BugMo11} for a large class of functions~$\psi$.  In particular, it was shown that, if $\psi$ is non-increasing and satisfies $\psi(x)=o(x^{-2})$, then
\[ \hdim\xct_\R(\psi)=\Bad_\R(\psi)=\hdim W_\R(\psi)=2/\lambda(\psi), \]
where
\begin{equation}\label{eq:lambda}
	\lambda(\psi):=\liminf_{x\to+\infty}\frac{-\log\psi(x)}{\log x}
\end{equation}
denotes \emph{the lower order} at infinity of the function~$1/\psi$. This notion appears naturally in Dodson's works~\cite{Dodso91,Dodso92} on the Hausdorff dimension of $\psi$-well approximable sets.

The situation becomes more subtle if we remove the assumption~$\psi(x)=o(x^{-2})$. Bugeaud~\cite{Bugea08} asked: If $\psi(x)=cx^{-2}$, is it possible that $\xct_\R(\psi)$ has positive Hausdorff dimension? Moreira~\cite{Morei18} gave a positive answer to this question for small~$c$ by showing that 
\[ \lim_{c\to 0}\hdim\xct_\R\bigl( x\mapsto cx^{-2} \bigr)\to 1. \]

For other works closely related to exact approximation, we refer to~\cite{Bugea08a,FraWh21,GerMo10,Zhang12}.

\subsection{Results}\label{ss:rt}

In this paper, we consider the analogous problem for complex numbers. Let 
\[ W(\psi):=\bigl \{z\in \dmn\colon |z-{p}/{q}|<\psi(|q|)\text{ for infinitely many } {p}/{q}\in \Q(i)\bigr\} \]
be the set of $\psi$-well approximable complex numbers and
\[ \xct(\psi):=W(\psi)\setminus\bigcup_{k\ge 2}W\bigl((1-1/k)\psi\bigr) \]
the set of complex numbers approximable to order $\psi$ but to no better order. Here 
\[ \dmn=\{x+iy\colon x,y\in[-1/2,1/2)\}. \] 

In~\cite{HeXi21}, we obtain the Hausdorff and packing dimensions of~$W(\psi)$. This is an analogue of the classical Jarn\'ik Theorem in real case.
\begin{thm}[{\cite[Theorem~1.2]{HeXi21}}]\label{t:well}
	If the function $\psi\colon(0,\infty)\to(0,\infty)$ is non-increasing and $\psi(x)\to0$ as $x\to+\infty$, then
	\[ \hdim W(\psi)=\min\bigl(4/\lambda(\psi),2\bigr)\quad \text{and}\quad \pdim W(\psi)=2. \] 
	Here $\lambda(\psi)$ is the lower order as in~\eqref{eq:lambda}.
\end{thm}

Based on the dimension results on the set~$W(\psi)$, we use the theory of complex continued fractions to determine the Hausdorff and packing dimensions of~$\xct(\psi)$ under the condition $\psi(x)=o(x^{-2})$.

\begin{thm}\label{t:ox2}
	Let $\psi\colon(0,\infty)\to(0,\infty)$ be a non-increasing function. If $\psi(x)=o(x^{-2})$, i.e., $x^2\psi(x)\to0$ as $x\to+\infty$, then
	\[ \hdim \xct(\psi)=4/\lambda(\psi) \quad \text{and} \quad \pdim \xct(\psi)=2. \]
\end{thm}

Just like in the case of real numbers, it seems very difficult to get accurate result for dimensions of~$\xct(\psi)$ without the condition $\psi(x)=o(x^{-2})$. However, by a careful analysis of admissible sequence appearing in the continued fraction of complex numbers, we can still obtain a nontrivial lower bound for Hausdorff dimension.

\begin{thm}\label{t:Ox2}
	Let $\psi\colon(0,\infty)\to(0,\infty)$ be a non-increasing function and 
	\[ \tau=\limsup_{x\to+\infty}x^2\psi(x). \]
	Suppose that $0<\tau\le\tau_0$, where $\tau_0$ is a positive constant,\footnote{To obtain a specific value of the constant~$\tau_0$ requires some more elaborate calculations, and it seems impossible to find the exact value of~$\hdim \xct(\psi)$ with our method under the condition of Theorem~\ref{t:Ox2}, even for small~$\tau$, so we will not pursue the quantification of~$\tau_0$ further.} then
	\[\hdim \xct(\psi)\ge 2-\tau/(1-2\tau).\]
\end{thm}

As for most of the metric problems in Diophantine approximation, the key to the proofs of Theorems~\ref{t:ox2} and~\ref{t:Ox2} is obtaining the lower bounds of dimensions. For this, we construct a Cantor-like set of large dimension and contained in~$\xct(\psi)$ by using cylinder sets of continued fractions.

Compared with the case of real numbers, the difficulty in proving Theorems~\ref{t:ox2} and~\ref{t:Ox2} is that the nature of complex continued fractions is more complicated than that of real continued fractions. More specifically, the following two points are the main obstacles. First, the cylinder sets of real continued fractions are all intervals on the line and their relative positional relationship is easier to analyze; while the cylinder sets of complex continued fractions have many different shapes, and the positional relationship between them is more complicated. Second, for real continued fractions (more precisely, the regular continued fractions), the cylinder sets are always nonempty; while, for complex continued fractions, there exist cylinder sets which are empty (or sequences which are not admissible, see Definition~\ref{d:admi}). These two points indicate that to construct the desired Cantor-like subset, we need to select the appropriate cylinder sets through more elaborate calculations and more delicate restrictions than the real number case.

\subsection{Notation}
Throughout this paper, we adhere to the following notation.
\begin{itemize}
	\item Sequences of numbers will be denoted by letters in boldface: $\bm a,\bm b,\dotsc$;
	\item for $\bm a=(a_1,\dots,a_n)$ and $\bm b=(b_1,\dots,b_m)$, write $\bm a^-=(a_1,\dots,a_{n-1})$, $\bm a^t=(a_n,a_{n-1},\dots,a_1)$ and $\bm{ab}=(a_1,\dots,a_n,b_1,\dots,b_m)$;
	\item denote by $\bm a\prec\bm b$ if $\bm a$ is a prefix of~$\bm b$;
	\item denote the empty word by~$\varnothing$ and empty set by~$\emptyset$;
	\item the open and closed ball with center~$z$ and radius~$r$ will be denoted by~$B(z,r)$ and $\overline B(z,r)$, respectively;
	\item the interior, closure, diameter, cardinal number and Lebesgue measure of~$A$ will be denoted by~$A^\circ$, $\overline A$, $|A|$, $\sharp A$ and~$\lm(A)$, respectively.
	\item  $\dmn=\{x+iy\colon x,y\in[-1/2,1/2)\}$ and $\dmn^*=\dmn\setminus\{0\}$;
	\item $I=\Z[i]\setminus\{0,\pm1,\pm i\}$ and  $I^*=\bigcup_{n\ge0} I^n$, with the convention $I^0=\{\varnothing\}$;
	\item For $M>0$, let $I_M=\{z\in I\colon|z|\le M\}$ and $I_M^*=\bigcup_{n\ge0}I_M^n$.
\end{itemize}

\subsection*{Acknowledgements}

The authors would like to thank Gerardo Gonz\'alez Robert for pointing out that the reference~\cite{EiINN19} contains a non-quantitative version of Lemma~\ref{l:badapp}.

\section{Hurwitz continued fraction}\label{sc:HCF}

\subsection{Basic properties of Hurwitz continued fraction}\label{ss:HCF}

The \emph{Hurwitz continued fraction} (HCF) was introduced by A.~Hurwitz~\cite{Hurwi87} in~1887.  To state its definition, we begin with some notation. The nearest Gaussian integer of a complex number~$z$, denoted by~$[z]$, is determined by $z-[z]\in\dmn$, where
\begin{equation}\label{eq:dmn}
	\dmn:=\{x+iy\colon x,y\in[-1/2,1/2)\}
\end{equation}
is \emph{the fundamental domain} of HCF. For $z\in\dmn^*:=\dmn\setminus\{0\}$ and $n\ge1$, define $a_n=[1/T^{n-1}(z)]$, where $T(z)=1/z-[1/z]$. The sequence $(a_n)$ is called the HCF partial quotients of~$z$, since it leads to the continued fraction expansion $z=[0;a_1,a_2,\dotsc]$, where
\[ [0;a_1,a_2,\dots,a_m]:= \cfrac{1}{a_1+\cfrac{1}{a_2+\cfrac{1}{\dots+a_m}}} \]
for finite sequence $(a_1,a_2,\dots,a_m)$ of numbers and
\[ [0;a_1,a_2,\dotsc]:=\lim_{m\to\infty}[0;a_1,a_2,\dots,a_m] \]
for infinite sequence. If $z\in\Q(i)$, then $T^m(z)=0$ for some $m\ge1$ and the partial quotients of~$z$ are $(a_1,\dots,a_m)$. Otherwise, $(a_n)$ is an infinite sequence. Moreover, since $[1/z]\notin\{0,\pm1,\pm i\}$ for $z\in\dmn^*$, we always have $a_n\in I$, where
\begin{equation}\label{eq:albe}
	I:=\Z[i]\setminus\{0,\pm1,\pm i\}
\end{equation}
is called \emph{the alphabet set} of HCF partial quotients.
We write $a_n(z)$ instead of $a_n$ when it is necessary to emphasis that $(a_n)$ is obtained from $z$.

Define $p(\varnothing)=0$ and $q(\varnothing)=1$ for the empty word $\varnothing$. Given a sequence $(a_1,\dots,a_n)$ of numbers with $n\ge1$, define
\begin{equation}\label{eq:Qpair}
	\begin{pmatrix}
		q(a_1,\dots,a_n) & q(a_1,\dots,a_{n-1}) \\ 
		p(a_1,\dots,a_n) & p(a_1,\dots,a_{n-1})
	\end{pmatrix}
	=\prod_{j=1}^{n}
	\begin{pmatrix}
		a_j & 1 \\ 
		1 & 0
	\end{pmatrix}.
\end{equation}
Here we regard $(a_1,\dots,a_{n-1})$ as the empty word~$\varnothing$ when $n=1$. For~$z\in\dmn^*$ with the HCF partial quotients $(a_n)$, let $p_n=p(a_1,\dots,a_n)$ and $q_n=q(a_1,\dots,a_n)$ with the convention $p_0=p(\varnothing)=0$ and $q_0=q(\varnothing)=1$. By~\eqref{eq:Qpair}, we have the recursive formulae:
\[p_n=a_np_{n-1}+p_{n-2}\quad\textrm{and}\quad q_n=a_nq_{n-1}+q_{n-2},\qquad \text{for}\ n\ge 2.\]
Moreover, it holds that ${p_n}/{q_n}=[0;a_1,a_2,\dots,a_n]$. We call the sequences $(p_n)$ and $(q_n)$ the \emph{$\mathcal Q$-pair} of~$z$ (or~$x$). 

The following lemma contains some basic facts of HCF, which can also be found in the references~\cite{DanNo14,EiINN19,Gonza18,Hensl06,HiaVa18,Hurwi87,Lakei73}.

\begin{lem}[{\cite[Lemma~2.2]{HeXi21}}]\label{l:props}
	Let $z\in\dmn^*$ and denote its HCF partial quotients by $(a_n)$. Let $(p_n)$ and $(q_n)$ be the $\mathcal Q$-pair of~$z$. For all $n\ge1$, the following statements hold.
	\begin{enumerate}[\upshape(a)]
		\item \label{le:mf} The mirror formula: $q_{n-1}/q_n=[0;a_n,a_{n-1},\dots,a_1]$. 
		\item \label{le:pq-pq} $q_np_{n-1}-q_{n-1}p_n=(-1)^n$. 
		\item \label{le:z-pq} $|z-p_n/q_n|=\bigl|q_n^2(a_{n+1}+T^{n+1}(z)+q_{n-1}/q_n)\bigr|^{-1}\le|q_n|^{-2}$. 
		\item \label{le:1<q} $1=|q_0|<|q_1|<|q_2|<\dotsb$.
		\item \label{le:q>q} $|q_{n+k}|\ge\phi^{\lfloor n/2\rfloor}|q_k|$ for $k\ge0$. In particular, $|q_n|\ge\phi^{\lfloor n/2\rfloor}$. Here $\phi=\frac{\sqrt5+1}{2}$ and $\lfloor x\rfloor$ denotes the greatest integer $\le x$. 
		\item \label{le:|q-|} Let $\bm a=(a_1,\dots,a_n)$, then $(|a_n|-1)|q(\bm a^-)|<|q(\bm a)|<(|a_n|+1)|q(\bm a^-)|$.
		\item \label{le:|qq|} Let $\bm a=(a_1,\dots,a_n)$, $\bm b=(a_{n+1},\dots,a_{n+k})$ with $k\ge1$. Then 
		\[ |q(\bm a)q(\bm b)|/5<|q(\bm{ab})|< 3|q(\bm a)q(\bm b)|. \]
	\end{enumerate}
\end{lem}

In the remainder of this subsection, we present some Diophantine approximation properties of HCF.

\begin{defn}[see~\cite{Lakei73}]
	Let $z$ be a complex number. A complex rational number $p/q$ is called a \emph{good approximation} to $z$, if for all $p',q'\in\Z[i]$,
	\[ |q'|\le|q|\implies |q'z-p'|\ge|qz-p|. \]
	Moreover, $p/q$ is called a \emph{best approximation} to $z$ if for any $p',q'\in \Z[i]$, 
	\[ |q'|<|q|\implies |q'z-p'|>|qz-p|. \]
\end{defn}

\begin{thm}[{\cite[Theorem~2]{Lakei73}}]\label{t:good}
	Every HCF convergent of a complex number $z$ is a good approximation to~$z$. Furthermore, for Lebesgue almost all $z\in\C$ every HCF convergent is a best approximation to~$z$.
\end{thm}

\begin{lem}\label{l:badapp}
	Let $z\in\mathfrak F\setminus \Q(i)$. If $p/q\in \Q(i)$ is nonzero and satisfies $|z-p/q|<1/(4|q|^2)$, then $p/q$ is a convergent of $z$.
\end{lem}
\begin{proof}
	Ei, Ito, Nakada and Natsui~\cite[Theorem~2]{EiINN19} drew the same conclusion by assuming $|z-p/q|<1/(L|q|^2)$ for some constant~$L$. But they did not give the explicit value of~$L$.
	
	Let $(p_n/q_n)_{n\ge0}$ be the convergent of~$z$. Since the sequence $(|q_n|)_{n\ge 0}$ is strictly increasing  and $|q_0|=1$, we may assume that $|q_{n-1}|\le |q|< |q_n|$ for some $n\ge1$. Suppose that $p/q\ne p_{n-1}/q_{n-1}$, we will prove that
	\begin{equation}\label{eq:1-t-t2}
		|z-p/q|\ge (1-t-t^2)/|q|^2
	\end{equation}
	and
	\begin{equation}\label{eq:t/2}
		|z-p/q|\ge t/(2|q|^2),
	\end{equation}
	where $t=|q|/|q_n|\in(0,1)$. This implies that $|z-p/q|\ge1/(4|q|^2)$, since
	\[ \min_{t\in(0,1)}\max(1-t-t^2,t/2)=1/4. \]
	Therefore, the lemma follows once we establish~\eqref{eq:1-t-t2} and~\eqref{eq:t/2}.
	
	We first prove~\eqref{eq:1-t-t2}. Note that Lemma~\ref{l:props}\eqref{le:pq-pq} gives $|p_n/q_n-p_{n-1}/q_{n-1}|=1/|q_{n-1}q_n|$ and  Lemma~\ref{l:props}\eqref{le:z-pq} gives $|z-p_n/q_n|\le|q_n|^{-2}$. Hence,
	\[ \begin{split}
		\biggl|z-\frac{p}{q}\biggr|&\ge\biggl|\frac{p_{n-1}}{q_{n-1}}-\frac{p}{q}\biggr|-\biggl|\frac{p_n}{q_n}-\frac{p_{n-1}}{q_{n-1}}\biggr| - \biggl| z-\frac{p_n}{q_n} \biggr|\ge \frac{1}{|qq_{n-1}|} -\frac{1}{|q_nq_{n-1}|} -\frac{1}{|q_n|^2}\\
		&= \frac{1}{|q|^2}\biggl( \biggl| \frac{q}{q_{n-1}} \biggr|\biggl( 1-\biggl| \frac{q}{q_n} \biggr| \biggr) -\biggl| \frac{q}{q_n} \biggr|^2 \biggr) \ge \frac{1-t-t^2}{|q|^2}.
	\end{split} \]
	
	Recall that $t=|q|/|q_n|\in(0,1)$. The proof of~\eqref{eq:t/2} is divided into two cases.
	
	\paragraph{Case~1: $|z-p_n/q_n|>1/(2|q_n|^2)$} By Theorem~\ref{t:good},
	\[ \biggl| z-\frac{p}{q} \biggr|\ge\frac{|q_n|}{|q|}\biggl| z-\frac{p_n}{q_n} \biggr|> \frac{|q_n|}{|q|}\frac{1}{2|q_n|^2} =\frac{t}{2|q|^2}. \]
	
	\paragraph{Case~2: $|z-p_n/q_n|\le1/(2|q_n|^2)$} We have
	\[ \biggl| z-\frac{p}{q} \biggr| \ge \biggl| \frac{p_n}{q_n}-\frac{p}{q} \biggr| -\biggl| z-\frac{p_n}{q_n} \biggr|\ge \frac{1}{|q_nq|}-\frac{1}{2|q_n|^2} =\frac{t-t^2/2}{|q|^2}>\frac{t}{2|q|^2}. \qedhere \]
\end{proof}

\subsection{Regular and full cylinders} \label{ss:RC}     

In this subsection, we discuss some properties related to the cylinder sets of HCF.

Recall that $I=\Z[i]\setminus\{0,\pm1,\pm i\}$ is the alphabet set of HCF partial quotients. A cylinder of level~$n$ is a subset of~$\mathfrak F$ taking the form
\[ \mathcal C(\bm u)=\bigl\{z\in\mathfrak F\colon a_1(z)=u_1,\dots,a_n(z)=u_n\bigr\}, \]
where $\bm u=(u_1,\dots,u_n)\in I^n$. We adopt the convention that $I^0=\{\varnothing\}$ and $\mathcal C(\varnothing)=\mathfrak F$. One can check that $\{\mathcal C(\bm u)\}_{\bm u\in I^n}$ forms a partition of~$\mathfrak F$ except for some points in $\Q(i)$ for each~$n\ge0$.

\begin{defn}\label{d:admi}
	A sequence $\bm u\in I^*:=\bigcup_{n\ge0}I^n$ is called admissible if $\cld(\bm u)\ne\emptyset$. 
\end{defn}

To study cylinders~$\cld(\bm u)$ of level~$n$, let $T_{\bm u}=(T^n|_{\cld(\bm u)})^{-1}$ with the convention $T_\varnothing=\mathrm{identity}$. We call $\dmn_{\bm u}=T^n\bigl(\cld(\bm u)\bigr)=T_{\bm u}^{-1}\bigl(\cld(\bm u)\bigr)$ the \emph{prototype set} of~$\cld(\bm u)$. Note that $T_{\bm u}\colon\dmn_{\bm u}\to\cld(\bm u)$ is the M\"obius transformation of the form 
\begin{equation}\label{eq:Tu}
	T_{\bm u}\colon z\mapsto\frac{p(\bm u^-)z+p(\bm u)}{q(\bm u^-)z+q(\bm u)}\quad\text{or equivalently}\ T_{\bm u}\colon [0;\bm v]\mapsto[0;\bm{uv}].
\end{equation}

\begin{exmp}[see~{\cite[page~8]{Gonza18}}]\label{e:dmn2}
	$\dmn_2=\dmn\setminus\overline B(-1,1)$, $\dmn_{-2}=\dmn\setminus B(1,1)$ and
	\[ \dmn_b=\dmn\quad\text{if $|b|\ge2\sqrt2$}. \]
\end{exmp}

The lemma below is useful to determine prototype sets.
\begin{lem}[{\cite[Lemma~2.3]{HeXi21}}]\label{l:image}
	Let $\bm u=(u_1,\dots,u_n)\in I^n$and $\bm v=(v_1,\dots,v_m)\in I^m$.
	\begin{enumerate}[\upshape(a)]
		\item \label{le:uv=u'v} If $\dmn_{\bm u}=\dmn_{\bm u'}$ for some sequence $\bm u'\in I^{n'}$, then $\dmn_{\bm{uv}}=\dmn_{\bm u'\bm v}$.
		\item \label{le:image} If $\cld(\bm v)\subset\dmn_{\bm u}$, then $\dmn_{\bm{uv}}=\dmn_{\bm v}$. In particular, if $\dmn_{\bm u}=\dmn_{\bm v}=\dmn$, then $\dmn_{\bm{uv}}=\dmn$.
	\end{enumerate}
\end{lem}

Now we give the definition of regular and full cylinders.

\begin{defn}[see~{\cite[\S1.2, \S3.1]{Gonza18}}]\label{d:regfl}
	We say a cylinder $\cld(\bm u)$ or a sequence $\bm u$ is regular if $\dmn_{\bm u}$ has nonempty interior, otherwise it is said to be irregular. Moreover, we say $\cld(\bm u)$ or $\bm u$ is full if $\dmn_{\bm u}=\dmn$.
\end{defn}

\begin{exmp}[see~{\cite[page~8]{Gonza18}}]\label{e:reg}
	All the cylinders of level~$1$ are regular. Moreover, a cylinder~$\cld(b)$ of level~$1$ is full if and only if $|b|\ge2\sqrt2$.
\end{exmp}

The following three lemmas coming from~\cite{HeXi21} are useful in estimating the dimension of Cantor-like sets which are constructed from cylinders. The first one quantifies the geometric size of the regular cylinders; the other two bound the cardinality of certain families of full sequences.

\begin{lem}[{\cite[Lemma~2.7]{HeXi21}}]\label{l:diamea}
	There is a constant~$0<\cst_0<1$ such that, for all regular cylinders~$\cld(\bm u)$,
	\[ \frac{\cst_0}{|q(\bm u)|^2}\le|\cld(\bm u)|\le\frac{2}{|q(\bm u)|^2}\quad\text{and}\quad
	\frac{\cst_0\pi}{|q(\bm u)|^4}\le\lm\bigl(\cld(\bm u)\bigr)\le\frac{\pi}{|q(\bm u)|^4}. \]
\end{lem}

\begin{lem}[{\cite[Lemma~2.12]{HeXi21}}]\label{l:OMQ}
	Given $M,Q>0$, let $I_M=\{z\in I\colon|z|\le M\}$ and $I_M^*=\bigcup_{n\ge0}I_M^n$, write
	\[ \Gamma_M(Q)=\bigl\{\bm u\in I_M^*\colon\text{$\bm u$ is full and $|q(\bm u^-)|<Q\le |q(\bm u)|$}\bigr\}. \]
	There exists a constant $\cst_1>1$ such that, for $M\ge12\cst_1$ and $Q\ge(M+1)^{5M}$, we have $\sharp\Gamma_M(Q)\ge Q^{4-2/M}$.
\end{lem}

\begin{lem}[{\cite[Lemma~2.13]{HeXi21}}]\label{l:CaQ}
	Let $\cst_1>1$ be the constant given by Lemma~\ref{l:OMQ}. Suppose that $M\ge12\cst_1$ and $Q>0$. Given a sequence $\bm w$, let
	\[ \Gamma_M^{\bm w}(Q)=\bigl\{\bm b\colon\bm{wb}\in\Gamma_M(Q)\bigr\}. \]
	Then
	\[ \sharp\Gamma_M^{\bm w}(Q)\le(M+1)^{24M}|q(\bm w)|^{-4+2/M}\cdot \sharp\Gamma_M(Q). \]
\end{lem}

We end this subsection with a lemma which gives the lower bound of the distance between $p(\bm u)/q(\bm u)$ and a point outside the cylinder $\cld(\bm u)$ when $\bm u$ is full.
\begin{lem}\label{l:>1/3q2}
	If $\bm u\in I^*$ is full, then for all $z\notin\cld(\bm u)$,
	\[ |z-{p(\bm u)}/{q(\bm u)}|>1/(3|q(\bm u)|^2).\]
\end{lem}
\begin{proof}
Since $\bm u$ is full, we have $\cld(\bm u)=T_{\bm u}(\dmn)$. Recall that
\[ T_{\bm u}\colon z\mapsto\frac{p(\bm u^-)z+p(\bm u)}{q(\bm u^-)z+q(\bm u)}\quad\text{and}\quad \dmn=\{x+iy\colon x,y\in[-1/2,1/2)\}. \]
Therefore, for all $z\notin\cld(\bm u)$,
\[ \begin{split}
	\biggl|z-\frac{p(\bm u)}{q(\bm u)}\biggr| 
	&\ge\inf_{z\in\partial\cld(\bm u)}\biggl |\frac{p(\bm u)}{q(\bm u)}-z \biggr | = \inf_{z\in\partial\dmn}\biggl |\frac{p(\bm u)}{q(\bm u)}-T_{\bm u}(z) \biggr | \\
	&=\inf_{z\in\partial\dmn}\biggl|\frac{p(\bm u)}{q(\bm u)}-\frac{p(\bm u^-)z+p(\bm u)}{q(\bm u^-)z+q(\bm u)}\biggr| =\inf_{z\in\partial\dmn}\biggl| \frac{z}{q(\bm u)\bigl(q(\bm u^-)z+q(\bm u)\bigr)} \biggr| \\ &>\inf_{z\in\partial\dmn}\frac{|z|}{(1+|z|)|q(\bm u)|^2} \ge \frac{1}{3|q(\bm u)|^2}.
\end{split} \]
The last inequality follows from the fact that $|z|\ge1/2$ for $z\in\partial\dmn$.
\end{proof}

\section{Proof of Theorem~\ref{t:ox2}}\label{sc:ox2}

In this section, we fix a non-increasing function~$\psi\colon(0,\infty)\to(0,\infty)$ with $\psi(x)=o(x^{-2})$ and denote $\lambda(\psi)$ briefly by~$\lambda$.

\subsection{Sketch of the proof} \label{ss:skt1}

Pick a rapidly increasing sequence~$(n_k)_{k\ge1}$ of positive numbers with $n_1=1$ such that $\psi(n_k)\approx n_k^{-\lambda}$ for $k\ge2$. We shall find full sequences $(\bm a_k)_{k\ge2}$ with $|q(\bm a_k)|\approx n_k$ such that, for some $b_k\in I(\bm a_k;k)$, $\bm a_kb_k$ is a prefix of $\bm a_{k+1}$, where $I(\bm a_k;k)$ is defined by~\eqref{eq:Iku}, which ensures the required exact approximation property~\eqref{eq:1-1/k}.

Now let $z$ be the unique number in $\bigcap_{k\ge2}\cld(\bm a_k)$. By Lemma~\ref{l:badapp} and inequality~\eqref{eq:1-1/k}, we can prove that $z\in\xct(\psi)$ under an additional mild restriction on~$(\bm a_k)$. Next, we use Lemmas~\ref{l:OMQ}, \ref{l:CaQ} and~\ref{l:ub} \eqref{le:NIuk} to show that there are sufficiently many choices of sequences $(\bm a_k)$ such that one can construct a sufficiently large Cantor-like subset of~$\xct(\psi)$ from such $(\bm a_k)$, and then obtain the desired lower bounds of the dimensions of~$\xct(\psi)$.

\subsection{Exact approximation}

The aim of this subsection is to prove Lemma~\ref{l:ub}, which is very helpful for us to find sufficiently many points in~$\xct(\psi)$.

Given a full sequence $\bm u\in I^*$ and $k\ge2$, we are interested in the full cylinders of the form $\cld(\bm ub)$ with $b\in I$ such that all the numbers $z$ in such a cylinder satisfy
\begin{equation}\label{eq:1-1/k}
	(1-1/k)\psi(|q(\bm u)|)\le|z-p(\bm u)/q(\bm u)|<\psi(|q(\bm u)|).
\end{equation}
We thus define
\begin{equation}\label{eq:Iku}
	I(\bm u;k):=\{b\in I\colon\text{$\bm ub$ is full and \eqref{eq:1-1/k} holds for all $z\in\cld(\bm ub)$}\}.
\end{equation}

\begin{lem}\label{l:ub}
	For each $k\ge2$, there is a constant $c_{\psi,k}$ dependent  on $\psi$ and $k$ such that, for full sequence $\bm u$ with $|q(\bm u)|\ge c_{\psi,k}$ and $b\in I(\bm u;k)$, the following statements hold.
	\begin{enumerate}[\upshape(a)]
		\item \label{le:z-p/q} $|z-p(\bm u)/q(\bm u)|<\bigl(6|q(\bm u)|^2\bigr)^{-1}$ for all $z\in\cld(\bm ub)$.
		\item \label{le:qub} $\bigl(2|q(\bm u)|^2\psi(|q(\bm u)|)\bigr)^{-1}\cdot |q(\bm u)|< |q(\bm ub)|< 3\bigl(|q(\bm u)|^2\psi(|q(\bm u)|)\bigr)^{-1}\cdot |q(\bm u)|$.
		\item \label{le:NIuk} $\bigl(k|q(\bm u)|^4\psi(|q(\bm u)|)^2\bigr)^{-1}< \sharp I(\bm u;k)< 3\pi\bigl(|q(\bm u)|^4\psi(|q(\bm u)|)^2\bigr)^{-1}$.
	\end{enumerate}
\end{lem}
\begin{proof}
Let $\rho:=\bigl(|q(\bm u)|^2\psi(|q(\bm u)|)\bigr)^{-1}$. We will prove~\eqref{le:z-p/q} and~\eqref{le:qub} under the hypothesis $\rho>6$, and prove~\eqref{le:NIuk} under the hypothesis $\rho>\rho_k$, where $\rho_k>6$ is a constant to be determined later. Since $\psi(x)=o(x^{-2})$, we can find a constant $c_{\psi,k}$ such that $\rho>\rho_k>6$ whenever $|q(\bm u)|\ge c_{\psi,k}$. This $c_{\psi,k}$ is desired.
	
(a) Suppose that $\rho>6$. It follows from~\eqref{eq:1-1/k} that, for all $z\in\cld(\bm ub)$,
\[ |z-p(\bm u)/q(\bm u)|<\psi(|q(\bm u)|)<\bigl(6|q(\bm u)|^2\bigr)^{-1}. \]

(b) Suppose that $\rho>6$. Let
\[ \begin{split}
	J_1 &=\{ b\in I\colon \rho+2 < |b| <(1-1/k)^{-1}\rho-2 \},\\
	J_2 &=\{ b\in I\colon \rho-2 < |b| <(1-1/k)^{-1}\rho+2 \}.
\end{split} \]
We claim that
\begin{equation}\label{eq:JIJ}
	J_1\subset I(\bm u;k)\subset J_2.
\end{equation}
We put the verification of the claim at the end of the whole proof. For now, note that, since $\rho>6$, it follows immediately from $I(\bm u;k)\subset J_2$ that 
\[ \rho/2+1 < |b| < 3\rho-1\quad
\text{for all $b\in I(\bm u;k)$}. \]
Applying Lemma~\ref{l:props} \eqref{le:|q-|}, we obtain
\[ (\rho/2)\cdot|q(\bm u)|<(|b|-1)|q(\bm u)|< |q(\bm ub)|< (|b|+1)|q(\bm u)|<3\rho\cdot |q(\bm u)|. \]

(c) The famous Gauss circle problem~\cite[Theorem~339]{HarWr08} in number theory asserts that the number of lattice points $N(R)$ inside the boundary of a circle of radius~$R$ with center at the origin is $\pi R^2+O(R)$. Therefore,
\[ \begin{split}
	\sharp J_1 &= N\bigl( (1-1/k)^{-1}\rho-2 \bigr) - N(\rho+2) \\
	&= \pi\bigl( (1-1/k)^{-1}\rho-2 \bigr)^2 - \pi(\rho+2)^2 + O(\rho) \\
	&=\frac{\pi(2k-1)}{(k-1)^2}\rho^2+O(\rho).
\end{split} \]
This argument also gives the same result for $\sharp J_2$. Thus, it follows from~\eqref{eq:JIJ} that 
\[ \sharp I(\bm u;k)=\frac{\pi(2k-1)}{(k-1)^2}\rho^2+O(\rho). \]
Since $1/k<{\pi(2k-1)}/{(k-1)^2}<3\pi$ for $k\ge2$, there exists a constant $\rho_k>6$ such that, for $\rho>\rho_k$, 
\[ \rho^2/k<\sharp I(\bm u;k)<3\pi\rho^2. \]
This completes the proof of~\eqref{le:NIuk}.

It remains to prove the claim~\eqref{eq:JIJ}. We first show that $J_1\subset I(\bm u;k)$. Suppose that $b\in J_1$, then $|b|>\rho+2>8$. By Example~\ref{e:reg}, $b$ is full. Since $\bm u$ is also full, so is $\bm ub$ (see Lemma~\ref{l:image} \eqref{le:image}). Now let $z\in\cld(\bm ub)$, to prove $b\in I(\bm u;k)$, it suffices to show that the inequality~\eqref{eq:1-1/k} holds. For this, note that Lemma~\ref{l:props} \eqref{le:z-pq} gives
\begin{equation}\label{eq:bpm2}
	|q(\bm u)|^{-2}(|b|+2)^{-1}<|z-p(\bm u)/q(\bm u)|<|q(\bm u)|^{-2}(|b|-2)^{-1}.
\end{equation}
Since $b\in J_1$, we have $\rho<|b|-2<|b|+2<(1-1/k)^{-1}\rho$. Consequently,
\[ (1-1/k)\rho^{-1}|q(\bm u)|^{-2}<|z-p(\bm u)/q(\bm u)|<\rho^{-1}|q(\bm u)|^{-2}. \]
This implies~\eqref{eq:1-1/k} since $\rho^{-1}|q(\bm u)|^{-2}=\psi(|q(\bm u)|)$.

Finally, we prove that $I(\bm u;k)\subset J_2$. Suppose that $b\in I(\bm u;k)$, then \eqref{eq:1-1/k} holds for $z\in\cld(\bm ub)$. Since $b$ is full, Example~\ref{e:reg} gives $|b|\ge2\sqrt{2}$, and so Lemma~\ref{l:props} \eqref{le:z-pq} implies~\eqref{eq:bpm2}. The inequalities~\eqref{eq:1-1/k} and~\eqref{eq:bpm2} combine to yield
\[ \rho-2 < |b| <(1-1/k)^{-1}\rho+2. \]
Therefore, $b\in J_2$.
\end{proof}

\subsection{Auxiliary families of full sequences} \label{ss:para1}

In this subsection, we introduce two auxiliary families of full sequences, which enables us to construct the desired Cantor-like subsets of~$\xct(\psi)$.

Fix $\epsilon>0$ sufficiently small and $M= \max(12\cst_1,2\epsilon^{-1})>12$, where $\cst_1>1$ is the constant in Lemma~\ref{l:OMQ}. Pick a rapidly increasing sequence~$(n_k)_{k\ge1}$ of positive numbers with $n_1=1$ such that, for $k\ge2$,
\begin{align}
	&n_k^{-\lambda-\epsilon/4}\le\psi(n_k)\le n_k^{-\lambda+\epsilon/4}; \label{eq:nkpsi}\\
	&\psi(n_k^{1-\epsilon/3})\le n_k^{-\lambda+\lambda\epsilon/2}; \label{eq:nkpsi'}\\
	&n_k^{\epsilon/6}\ge 9(M+1)n_{k-1}^\lambda>100; \label{eq:nkeps}\\
	&n_k\ge\max\bigl((M+1)^{5M/(1-\epsilon/4)},c_{\psi,k}^{1/(1-\epsilon/3)},k^{1/(\lambda\epsilon)}\bigr). \label{eq:nk>}
\end{align}
Here $c_{\psi,k}$ is the constant given by Lemma~\ref{l:ub}.

For $k\ge2$, set $Q_k=n_k^{1-\epsilon/4}$. It follows from~\eqref{eq:nk>} that
\begin{equation}\label{eq:Qk}
	Q_k=n_k^{1-\epsilon/4}\ge(M+1)^{5M}\quad\text{for $k\ge2$}.
\end{equation}
Let $\Lambda_1=\Lambda_1'=\{\varnothing\}$. For $k\ge2$, define $\Lambda_k$ and $\Lambda_k'$ inductively by
\begin{equation}\label{eq:Lam}
\begin{split}
	\Lambda_k &=\{\bm a'\bm u\colon\bm a'\in\Lambda_{k-1}', \bm u\in\Gamma_M(Q_k)\},\\
	\Lambda_k'&=\{\bm ab\colon\bm a\in\Lambda_k,b\in I(\bm a;k)\}.
\end{split}
\end{equation}
We remark that all $\bm a\in\Lambda_k$ are full (see Lemma~\ref{l:Lam} \eqref{le:LamF}), thus $I(\bm a;k)$ and $\Lambda_k'$ are well defined.

\begin{lem}\label{l:Lam}
Let $\cst_0$ be the constant in Lemma~\ref{l:diamea}. For $k\ge2$, the following statements hold.
\begin{enumerate}[\upshape(a)]
	\item \label{le:LamF} All the sequences in $\Lambda_k$ and $\Lambda_k'$ are full.

	\item \label{le:qLam} For $\bm a\in\Lambda_k$, $n_k^{1-\epsilon/3}\le|q(\bm a)|\le n_k/3$. Moreover,
	\[ |\cld(\bm a)|\le 2n_k^{-2+\epsilon},\quad \lm\bigl(\cld(\bm a)\bigr)\ge\cst_0\pi n_k^{-4}. \]
	
	\item \label{le:qLam'} For $\bm a'\in\Lambda_k'$, $n_k^{\lambda-1-\lambda\epsilon} \le|q(\bm a')| \le n_k^{\lambda-1+\epsilon}$. Moreover,
	\[ \cst_0n_k^{-2\lambda+2-2\epsilon}\le|\cld(\bm a')|\le2n_k^{-2\lambda+2+2\lambda\epsilon},\quad \lm\bigl(\cld(\bm a')\bigr)\ge\cst_0\pi n_k^{-4\lambda+4-4\epsilon}. \]
	
	\item \label{le:dtcld} For distinct $\bm a_1,\bm a_2\in\Lambda_k$, $b_1\in I(\bm a_1;k)$ and $b_2\in I(\bm a_2;k)$, the distance
	\[ \dist\bigl(\cld(\bm a_1b_1),\cld(\bm a_2b_2)\bigr)\ge n_k^{-2}. \]
	
	\item \label{le:dd<psi} Let $\bm a'=\bm ab\in\Lambda_k'$ with $\bm a\in\Lambda_k$ and $b\in I(\bm a;k)$, then for $z\in\cld(\bm a')$,
	\[ (1-1/k)\psi(|q(\bm a)|)\le|z-p(\bm a)/q(\bm a)|<\psi(|q(\bm a)|). \]

	\item \label{le:NI(a;k)} For $\bm a\in\Lambda_k$, we have $n_k^{2\lambda-4-2\lambda\epsilon}\le \sharp I(\bm a;k)\le n_k^{2\lambda-4+2\epsilon}$.
	\end{enumerate}
\end{lem}
\begin{proof}
(a) According to Lemma~\ref{l:image} \eqref{le:image}, if both $\bm u$ and $\bm v$ are full, then so is $\bm{uv}$. Therefore,using the definition of $\Gamma_M(Q_k)$ (see Lemma~\ref{l:OMQ}) and $I(\bm a;k)$ (see~\eqref{eq:Iku}), we can prove~\eqref{le:LamF} by induction on~$k$.

(b) (c) It suffices to prove the bounds for~$|q(\bm a)|$ and $|q(\bm a')|$, then the estimations for the diameters and Lebesgue measures of cylinders can be derived by Lemma~\ref{l:diamea}. 

We proceed by induction on~$k$. For $k=1$, the bounds for~$|q(\bm a)|$ and $|q(\bm a')|$ hold since $\Lambda_1=\Lambda_1'=\{\varnothing\}$ and $q(\varnothing)=n_1=1$. Now suppose this hold for $k-1$. Let $\bm a=\bm a'\bm u\in\Lambda_k$ with $\bm a'\in\Lambda_{k-1}'$ and $\bm u\in\Gamma_M(Q_k)$. Then by the definition of~$\Gamma_M(Q_k)$ (see Lemma~\ref{l:OMQ}), Lemma~\ref{l:props} \eqref{le:|q-|} and~\eqref{eq:Qk},
\begin{equation}\label{eq:quqk}
	n_k^{1-\epsilon/4}=Q_k\le|q(\bm u)|<(M+1)Q_k=(M+1)n_k^{1-\epsilon/4}.
\end{equation}
Combining Lemma~\ref{l:props} \eqref{le:|qq|} with the inequality~\eqref{eq:quqk}, we have
\[ |q(\bm a)|>|q(\bm a')q(\bm u)|/5\ge|q(\bm u)|/5\ge n_k^{1-\epsilon/4}/5>n_k^{1-\epsilon/3}. \]
The last inequality follows from~\eqref{eq:nkeps}, which asserts that $n_k^{\epsilon/12}>5$. For upper bound, combining Lemma~\ref{l:props} \eqref{le:|qq|}, \eqref{eq:quqk} with the induction hypothesis, we obtain
\[ \begin{split}
	|q(\bm a)|&<3|q(\bm a')q(\bm u)|\le 3n_{k-1}^{\lambda-1+\epsilon}\cdot(M+1)n_k^{1-\epsilon/4}\\
	&\le\bigl(9(M+1)n_{k-1}^\lambda\bigr)\cdot n_k^{1-\epsilon/4}/3\le n_k/3.
\end{split} \]
The last inequality follows from~\eqref{eq:nkeps}, which asserts that $9(M+1)n_{k-1}^\lambda\le n_k^{\epsilon/6}$.

Now we turn to the bounds for $|q(\bm a')|$, where $\bm a'=\bm ab\in\Lambda_k'$ with $\bm a\in\Lambda_k$ and $b\in I(\bm a;k)$. Recall that we have proved $n_k^{1-\epsilon/3}<|q(\bm a)|\le n_k/3$. Since $\psi(x)$ is non-increasing, this inequality together with \eqref{eq:nkpsi} and \eqref{eq:nkpsi'} gives
\begin{equation}\label{eq:psiqa}
	n_k^{-\lambda-\epsilon/4}\le\psi(n_k)\le \psi(|q(\bm a)|)\le\psi(n_k^{1-\epsilon/3})\le n_k^{-\lambda +\lambda\epsilon/2}.
\end{equation}
Since $\bm a$ is full and $|q(\bm a)|>n_k^{1-\epsilon/3}\ge c_{\psi,k}$ (see~\eqref{eq:nk>}), Lemma~\ref{l:ub} \eqref{le:qub} is applicable. Together with~\eqref{eq:psiqa}, we obtain
\[ \begin{split}
	|q(\bm a')|&=|q(\bm ab)|>\bigl(2|q(\bm a)|^2\psi(|q(\bm a)|)\bigr)^{-1}\cdot|q(\bm a)|\\
	&=\bigl(2|q(\bm a)|\psi(|q(\bm a)|)\bigr)^{-1}\ge 3n_k^{-1}\cdot n_k^{\lambda-\lambda\epsilon/2}/2> n_k^{\lambda-1-\lambda\epsilon},
\end{split} \]
and
\[ \begin{split}
	|q(\bm a')|&=|q(\bm ab)|<3\bigl(|q(\bm a)|^2\psi(|q(\bm a)|)\bigr)^{-1}\cdot|q(\bm a)|\\
	&=3\bigl(|q(\bm a)|\psi(|q(\bm a)|)\bigr)^{-1}< 3n_k^{-1+\epsilon/3}\cdot n_k^{\lambda+\epsilon/4}< n_k^{\lambda-1+\epsilon}.
\end{split} \]

(d) Since $\bm a_1$ is full (by~\eqref{le:LamF}) and $\cld(\bm a_1)\cap\cld(\bm a_2b_2)=\emptyset$ (by the condition $\bm a_1\ne\bm a_2$), Lemma~\ref{l:>1/3q2} gives
\begin{equation}\label{eq:za2}
	|z-p(\bm a_1)/q(\bm a_1)|>\bigl(3|q(\bm a_1)|^2\bigr)^{-1}\quad \text{for $z\in\cld(\bm a_2b_2)$}.
\end{equation}
On the other hand, since $|q(|\bm a_1|)\ge n_k^{1-\epsilon/3}\ge c_{\psi,k}$ (by~\eqref{le:qLam} and~\eqref{eq:nk>}), we can apply Lemma~\ref{l:ub} \eqref{le:z-p/q} to get
\[ |z-p(\bm a_1)/q(\bm a_1)|<\bigl(6|q(\bm a_1)|^2\bigr)^{-1}\quad\text{for all $z\in\cld(\bm a_1b_1)$}. \]
This together with~\eqref{eq:za2} implies that
\[ \dist\bigl(\cld(\bm a_1b_1),\cld(\bm a_2b_2)\bigr) >\bigl(6|q(\bm a_1)|^2\bigr)^{-1}> n_k^{-2}. \]
The last inequality follows from $|q(\bm a_1)|\le n_k/3$ (see~\eqref{le:qLam}).

(e) By the definition of $I(\bm a;k)$ (see~\eqref{eq:Iku}), this follows from~\eqref{eq:1-1/k}.

(f) By~\eqref{le:qLam}, $n_k^{1-\epsilon/3}\le|q(\bm a)|<n_k$. This together with~\eqref{eq:nkpsi}--\eqref{eq:nk>} gives
\[ k|q(\bm a)|^4\psi(|q(\bm a)|)^2\le n_k^{\lambda\epsilon}\cdot n_k^4\cdot\psi(n_k^{1-\epsilon/3})^2\le n_k^{4-2\lambda+2\lambda\epsilon}, \]
and 
\[ (3\pi)^{-1}|q(\bm a)|^4\psi(|q(\bm a)|)^2\ge n_k^{-\epsilon/6}\cdot n_k^{4-4\epsilon/3}\psi(n_k)^2\ge n_k^{4-2\lambda-2\epsilon}. \]
Here, we also use the fact that $\psi$ is non-increasing. Combining Lemma~\ref{l:ub} \eqref{le:NIuk} with the above two inequalities, we obtain~\eqref{le:NI(a;k)}.
\end{proof}

\subsection{Lower bound for dimensions}
Define
\begin{equation}\label{eq:E}
	E=\bigcap_{k\ge1}\bigcup_{\bm a\in\Lambda_k}\cld(\bm a).
\end{equation}
To obtain lower bounds for dimensions of~$E$, we shall construct a probability measure~$\mu$ on $E$. For this purpose, we begin with a claim that 
\begin{equation}\label{eq:Ecapa}
	E\cap\cld(\bm a)\ne\emptyset\quad\text{for all $\bm a\in\Lambda_k$ with $k\ge1$}.
\end{equation}
To see this, pick $\bm a_{k+j}\in\Lambda_{k+j}$ for $j\ge1$ such that $\bm a\prec\bm a_{k+1}\prec\bm a_{k+2}\prec\cdots$, where $\bm u\prec\bm v$ means that $\bm u$ is a prefix of~$\bm v$. Since all the sequences $\bm a_{k+j}$ are full (by Lemma~\ref{l:Lam} \eqref{le:LamF}), the intersection $\bigcap_{k\ge1}\cld(\bm a_{k+j})$ contains the number~$z=\lim_{j\to\infty}[0;\bm a_{k+j}]$. Clearly, we have $z\in E\cap\cld(\bm a)$.

With \eqref{eq:Ecapa} in hand, we can define $\mu$ on~$E$ inductively. For $k=1$, define
\[ \mu(E\cap\cld(\varnothing))=\mu(E\cap\dmn)=\mu(E)=1. \] 
Given a sequence $\hat{\bm a}\in\Lambda_{k-1}$, suppose now that we have defined $\mu(E\cap\cld(\hat{\bm a}))$. For all $\bm a\in\Lambda_k^{\hat{\bm a}}:=\{\bm a\in\Lambda_k\colon\hat{\bm a}\prec\bm a\}$, define
\begin{equation}\label{eq:mua}
	\mu(E\cap \cld(\bm a))=\frac{\mu(E\cap\cld(\hat{\bm a}))}{\sharp\Lambda_k^{\hat{\bm a}}}.
\end{equation}
This ensures $\mu(E\cap\cld(\hat{\bm a}))=\sum_{\bm a\in\Lambda_k^{\hat{\bm a}}}\mu(E\cap\cld(\bm a))$. Then by Carath\'eodory's extension theorem, such a measure~$\mu$ does exist.

\begin{lem}\label{l:mua}
	For all $k\ge1$, $\bm a\in\Lambda_k$ and $\bm a'\in\Lambda_k'$,
	\[ \mu(\cld(\bm a))\le n_k^{-4+2\epsilon} \quad\text{and}\quad \mu(\cld(\bm a'))\le n_k^{-2\lambda+3\lambda\epsilon}. \]
\end{lem}
\begin{proof}
	For $k=1$, there is nothing to prove since $n_1=1$. 
	
	For $\bm a\in\Lambda_k$ with $k\ge2$, let $\hat{\bm a}$ be the unique prefix of~$\bm a$ in $\Lambda_{k-1}$. By the definition of~$\Lambda_k$ (see~\eqref{eq:Lam}), Lemma~\ref{l:OMQ}, \eqref{eq:Qk} and the fact $2/M\le\epsilon$, we have
	\[ \sharp\Lambda_k^{\hat{\bm a}}\ge\sharp\Gamma_M(Q_k)\ge Q_k^{4-2/M}\ge(n_k^{1-\epsilon/4})^{4-\epsilon} >n_k^{4-2\epsilon}. \]
	Since $E$ is a set of full measure, \eqref{eq:mua} together with the above inequality gives
	\[ \mu(\cld(\bm a))=\mu(\cld(\hat{\bm a}))\cdot(\sharp\Lambda_k^{\hat{\bm a}})^{-1}\le (\sharp\Lambda_k^{\hat{\bm a}})^{-1}\le n_k^{-4+2\epsilon}. \]
	
	Now for $\bm a'\in\Lambda_k'$ with $k\ge2$, let $\bm a$ be the unique prefix of~$\bm a'$ in~$\Lambda_k$. By the definition of $\Lambda_k'$ (see~\eqref{eq:Lam}) and~\eqref{eq:mua},
	\[ \begin{split}
		\mu(\cld(\bm a'))&=\sum_{\bm a'\prec\check{\bm a}\in\Lambda_{k+1}}\mu(\cld(\check{\bm a}))=\frac{\mu(\cld(\bm a))}{\sharp\Lambda_{k+1}^{\bm a}}\cdot\sharp\Gamma_M(Q_{k+1})\\
		&=\mu(\cld(\bm a))\cdot(\sharp I(\bm a;k))^{-1}\le n_k^{-4+2\epsilon}\cdot n_k^{4-2\lambda+2\lambda\epsilon}\le n_k^{-2\lambda+3\lambda\epsilon}.
	\end{split} \]
	The last two inequalities follow from Lemma~\ref{l:Lam} \eqref{le:NI(a;k)} and $\lambda\ge2$, respectively.
\end{proof}

\begin{lem}\label{l:meaB}
	There exists a constant~$\cst$ with the following property. Let $z\in E$ and $0<r<n_2^{-2}$. Suppose that $n_{k+1}^{-2}\le r<n_k^{-2}$ for some $k\ge2$, then
	\begin{enumerate}[\upshape(a)]
		\item \label{le:rnk+1} $\mu\bigl(B(z,r)\bigr)\le\cst r^{2-2\epsilon}$ if $n_{k+1}^{-2}\le r<n_{k+1}^{-2+\epsilon}$;
		\item $\mu\bigl(B(z,r)\bigr)\le\cst r^{\lambda/(\lambda-1)-6\epsilon}$ if $n_{k+1}^{-2+\epsilon}\le r<n_k^{-2\lambda+2-2\epsilon}$;
		\item $\mu\bigl(B(z,r)\bigr)\le\cst r^{4/\lambda-6\lambda\epsilon}$ if $n_k^{-2\lambda+2-2\epsilon}\le r<n_k^{-2}$.
	\end{enumerate}
\end{lem}
\begin{proof}
	(a) For $z\in E$, let 
	\[ \Lambda_{k+1}(z)=\bigl\{\bm a\in\Lambda_{k+1}\colon\cld(\bm a)\cap B(z,r)\ne\emptyset\bigr\}. \]
	By Lemma~\ref{l:Lam}~\eqref{le:qLam}, $|\cld(\bm a)|\le2n_{k+1}^{-2+\epsilon}$ for $\bm a\in\Lambda_{k+1}$. Combining with the condition $r<n_{k+1}^{-2+\epsilon}$ gives
	\[ \bm a\in\Lambda_{k+1}(z)\implies \cld(\bm a)\subset B(z,3n_{k+1}^{-2+\epsilon}). \]
	Hence, by the lower bound for $\lm\bigl(\cld(\bm a)\bigr)$ in Lemma~\ref{l:Lam}~\eqref{le:qLam},
	\[ \sharp\Lambda_{k+1}(z)\le\frac{\lm(B(z,3n_{k+1}^{-2+\epsilon}))}{\min_{\bm a\in\Lambda_{k+1}}\lm(\cld(\bm a))} \le\frac{9\pi n_{k+1}^{-4+2\epsilon}}{\cst_0\pi n_{k+1}^{-4}} =9\cst_0^{-1}n_{k+1}^{2\epsilon}. \]
	This together with Lemma~\ref{l:mua} gives
	\[ \mu\bigl(B(z,r)\bigr)\le\sum_{\bm a\in\Lambda_{k+1}(z)}\mu\bigl(\cld(\bm a)\bigr)  \le{9\cst_0^{-1}n_{k+1}^{2\epsilon}}\cdot{n_{k+1}^{-4+2\epsilon}} =9\cst_0^{-1}n_{k+1}^{-4+4\epsilon}. \]
	Finally, we have $\mu\bigl(B(z,r)\bigr) \le9\cst_0^{-1}r^{2-2\epsilon}$ since $n_{k+1}^{-2}\le r$.
	
	(b) By Lemma~\ref{l:Lam}~\eqref{le:qLam} and~\eqref{le:qLam'}, the condition on~$r$ means that, for all $\bm a\in\Lambda_{k+1}$ and $\bm a'\in\Lambda_k'$,
	\[ |q(\bm a)|^{-2}< r<|q(\bm a')|^{-2}. \]
	Given $z\in E$, we want to find all the cylinders that intersect $B(z,r)$ and of diameter close to~$r$. The bounds above suggest considering the set $\Lambda_{k+1}^-(z)$ consisting of sequences with the form $\bm a'\bm w$ such that $\bm a'\in\Lambda_k'$, $\bm a'\bm w$ is a prefix of a sequence $\bm a\in\Lambda_{k+1}$ and
	\[ \cld(\bm a'\bm w)\cap B(z,r)\ne\emptyset,\quad |q(\bm a'\bm w)|^{-2}\le r<|q(\bm a'\bm w^-)|^{-2}. \]

	For $\bm a'\bm w\in\Lambda_{k+1}^-(z)$, by Lemma~\ref{l:diamea}, $|\cld(\bm a'\bm w)|\le2|q(\bm a'\bm w)|^{-2}\le2r$, and so $\cld(\bm a'\bm w)\subset B(z,3r)$. Using Lemma~\ref{l:diamea} again, we have
	\[ \lm\bigl(\cld(\bm a'\bm w)\bigr)\ge\frac{\cst_0\pi }{|q(\bm a'\bm w)|^4} >\frac{\cst_0\pi }{(M+1)^4|q(\bm a'\bm w^-)|^4} >\frac{\cst_0\pi r^2}{(M+1)^4}. \]
	Here the second inequality follows from Lemma~\eqref{l:props}~\eqref{le:|q-|}. Therefore,
	\begin{equation}\label{eq:NL-}
		\sharp\Lambda_{k+1}^-(z)\le\frac{\lm\bigl(B(z,3r)\bigr)}{\min_{\bm a'\bm w\in\Lambda_{k+1}^-(z)}\lm(\cld(\bm a'\bm w))} \le\frac{9\pi r^2}{(M+1)^{-4}\cst_0\pi r^2}=9(M+1)^4\cst_0^{-1}.
	\end{equation}
	
	We now turn to bound $\mu\bigl(\cld(\bm a'\bm w)\bigr)$ from above. Note that all cylinders $\cld(\bm a)$ contained by $\cld(\bm a')$ with $\bm a\in\Lambda_{k+1}$ have same $\mu$ measure. By~\eqref{eq:Lam}, 
	\begin{align*}
		&\sharp\{\cld(\bm a)\subset\cld(\bm a'\bm w)\colon\bm a\in\Lambda_{k+1}\}=\sharp\bigl\{\bm b\colon\bm{wb}\in\Gamma_M(Q_{k+1})\bigr\}= \sharp\Gamma_M^{\bm w}(Q_{k+1});\\
		&\sharp\{\cld(\bm a)\subset\cld(\bm a')\colon\bm a\in\Lambda_{k+1}\}=\sharp\Gamma_M(Q_{k+1}).
	\end{align*}
	This together with Lemma~\ref{l:CaQ} and Lemma~\ref{l:mua} gives
	\begin{equation}\label{eq:mua'w0}
		\mu\bigl(\cld(\bm a'\bm w)\bigr)=\frac{\sharp\Gamma_M^{\bm w}(Q_{k+1})}{\sharp\Gamma_M(Q_{k+1})}\cdot\mu(\cld(\bm a'))\le\frac{(M+1)^{24M}\cdot n_k^{-2\lambda+3\lambda\epsilon}}{|q(\bm w)|^{4-2/M}}.
	\end{equation}
	Combining the fact $2/M\le\epsilon$, Lemma~\ref{l:props}~\eqref{le:|qq|} with Lemma~\ref{l:Lam}~\eqref{le:qLam'}, we have
	\[ |q(\bm w)|^{-(4-2/M)}\le\bigl(3|q(\bm a')|\cdot|q(\bm a'\bm w)|^{-1}\bigr)^{4-\epsilon} \le(3n_k^{\lambda-1+\epsilon}r^{1/2})^{4-\epsilon} < 81n_k^{4\lambda-4+4\epsilon}r^{2-\epsilon}.  \]
	Write $c=81(M+1)^{24M}$, it follows from~\eqref{eq:mua'w0} and the inequality above that
	\begin{equation}\label{eq:mua'w}
		\mu\bigl(\cld(\bm a'\bm w)\bigr)\le cr^{2-\epsilon}n_k^{2\lambda-4+3\lambda\epsilon+4\epsilon} \le cr^{2-\epsilon-\frac{2\lambda-4+3\lambda\epsilon+4\epsilon}{2\lambda-2}} \le cr^{\frac{\lambda}{\lambda-1}-6\epsilon}.
	\end{equation}
	Here we use the condition $r<n_k^{-2\lambda+2-2\epsilon}<n_k^{-2\lambda+2}$ and the fact $\lambda\ge2$.
	
	Since $\mu\bigl(B(z,r)\bigr)\le\sum_{\bm a'\bm w\in\Lambda_{k+1}^-(z)}\mu\bigl(\cld(\bm a'\bm w)\bigr)$, the conclusion follows from~\eqref{eq:NL-} and~\eqref{eq:mua'w}.
	
	(c) For $z\in E$, let
	\[ \Lambda_k'(z)=\{\bm a'\in\Lambda_k'\colon \cld(\bm a')\cap B(z,r)\ne\emptyset\}. \]
	In order to get the best upper bound for $\mu\bigl(B(z,r)\bigr)$, we need to use two methods to bound $\sharp\Lambda_k'(z)$ from above. The first one is the same as that in the proof of~\eqref{le:rnk+1}. By Lemma~\ref{l:Lam}~\eqref{le:qLam'}, $|\cld(\bm a')|\le2n_k^{-2\lambda+2+2\lambda\epsilon}$ for $\bm a'\in\Lambda_k'$. Since $r\ge n_k^{-2\lambda+2-2\epsilon}$ and $\lambda\ge 2$, we have
	\[ \bm a'\in\Lambda_k'(z)\implies \cld(\bm a')\subset B(z,3n_k^{3\lambda\epsilon}r). \]
	By the lower bound for $\lm\bigl(\cld(\bm a')\bigr)$ in Lemma~\ref{l:Lam}~\eqref{le:qLam'},
	\[ \sharp\Lambda_k'(z)\le\frac{\lm(B(z,3n_k^{3\lambda\epsilon}r))}{\min_{\bm a'\in\Lambda_k'}\lm(\cld(\bm a'))} \le\frac{9\pi n_k^{6\lambda\epsilon}r^2}{\cst_0\pi n_k^{-4\lambda+4-4\epsilon}} \le 9\cst_0^{-1}n_k^{4\lambda-4+8\lambda\epsilon}r^2. \]
	
	The second method is to apply Lemma~\ref{l:Lam}~\eqref{le:dtcld}. Since $r<n_k^{-2}$, it implies that, if $\bm a_1',\bm a_2'\in\Lambda_k'(z)$, then there is a sequence $\bm a\in\Lambda_k$ such that $\bm a_1'=\bm ab_1$ and $\bm a_2'=\bm ab_2$ with $b_1,b_2\in I(\bm a;k)$. Hence, by Lemma~\ref{l:Lam}~\eqref{le:NI(a;k)}, we have 
	\[ \sharp\Lambda_k'(z)\le\sharp I(\bm a;k)\le n_k^{2\lambda-4+2\epsilon}. \]
	Recall that $\cst_0<1$ and $\lambda\ge2$. From the above two upper bounds for $\sharp\Lambda_k'(z)$,
	\[ \sharp\Lambda_k'(z)\le (9\cst_0^{-1}n_k^{4\lambda-4+8\lambda\epsilon}r^2)^{2/\lambda}(n_k^{2\lambda-4})^{1-2/\lambda}\le 9\cst_0^{-1}r^{4/\lambda}n_k^{2\lambda+8\lambda\epsilon}. \]
	This together with Lemma~\ref{l:mua} and the condition $r<n_k^{-2}$ implies
	\[ \mu\bigl(B(z,r)\bigr)\le{\sharp\Lambda_k'(z)}\cdot n_k^{-2\lambda+3\lambda\epsilon} \le 9\cst_0^{-1}r^{4/\lambda}n_k^{11\lambda\epsilon} \le 9\cst_0^{-1}r^{4/\lambda-6\lambda\epsilon}. \qedhere \]
\end{proof}

\begin{lem}\label{l:E<}
	$E\subset \xct(\psi)$.
\end{lem}
\begin{proof}
	Let $z\in E$. To prove $z\in\xct(\psi)$, by the fact $\psi(x)=o(x^{-2})$ and Lemma~\ref{l:badapp}, we only need to consider $|z-p/q|$ for $p/q\in\Q(i)$ is a convergent of~$z$. 
	
	Let $\bm a$ be a prefix of the HCF partial quotients of~$z$, then $z\in\cld(\bm ab)$ for some $b\in I$. By the definition of~$E$ (see~\eqref{eq:E} and~\eqref{eq:Lam}), we know that $b\in I(\bm a;k)$ if $\bm a\in\Lambda_k$ for some $k\ge2$; otherwise $|b|\le M$. It follows from Lemma~\ref{l:Lam} \eqref{le:dd<psi} and Lemma~\ref{l:props} \eqref{le:z-pq} that
	\[ |z-p(\bm a)/q(\bm a)| \begin{cases}
		\in \bigl[(1-1/k)\psi(|q(\bm a)|),\psi(|q(\bm a)|)\bigr), & \text{if $\bm a\in\Lambda_k$, $k\ge2$};\\
		\ge(M+2)^{-1}|q(\bm a)|^{-2}, &\text{otherwise}.
	\end{cases} \]
	This implies that $z\in\xct(\psi)$ since $\psi(x)=o(x^{-2})$.
\end{proof}

\begin{proof}[Proof of Theorem~\ref{t:ox2}]
	By the condition $\psi(x)=o(x^{-2})$, we have $\lambda\ge2$. Consequently, $2\ge\lambda/(\lambda-1)\ge 4/\lambda$. So Lemma~\ref{l:meaB} implies that, for all $z\in E$,
	\[ \liminf_{r\to0}\frac{\log\mu(B(z,r))}{\log r}\ge4/\lambda-6\lambda\epsilon, \quad \limsup_{r\to0}\frac{\log\mu(B(z,r))}{\log r}\ge2-2\epsilon. \]
	Hence, the mass distribution principle (\cite[Proposition~2.3]{Falco97}) gives 
	\[ \hdim E\ge4/\lambda-6\lambda\epsilon \quad\text{and}\quad \pdim E\ge2-2\epsilon. \]
	Since $\xct(\psi)\supset E$ (by Lemma~\ref{l:E<}) and $\epsilon$ is arbitrary, we get 
	\[ \hdim \xct(\psi)\ge4/\lambda \quad\text{and}\quad \pdim \xct(\psi)\ge2.  \]
	
	On the other hand, since $\xct(\psi)\subset W(\psi)$ and $\lambda\ge2$, Theorem~\ref{t:well} implies
	\[ \hdim \xct(\psi)\le4/\lambda \quad\text{and}\quad \pdim \xct(\psi)\le2. \qedhere \]
\end{proof}

\section{Proof of Theorem~\ref{t:Ox2}}\label{sc:Ox2}

\subsection{Sketch of the proof}

Let $\psi\colon(0,\infty)\to(0,\infty)$ be a non-increasing function such that $0<\tau=\limsup_{x\to+\infty}x^2\psi(x)$ is small enough. Pick a rapidly increasing sequence~$(n_k)_{k\ge1}$ of positive numbers with $n_1=1$ such that $\psi(n_k)\approx\tau n_k^{-2}$ for $k\ge2$. We shall find full sequences $\bm\varpi_2\prec\bm\varpi_3\prec\dotsb$ of the form $\bm\varpi_k=\bm w_k\bm a_kt_k\bm b_k$, where $\bm w_k$, $\bm a_k$, $t_k$ and $\bm b_k$ will be determined later to ensure the exact approximation property. Roughly speaking, we require that the following conditions hold. For $k\ge2$, 
\begin{enumerate}[(i)]
	\item $|q(\bm w_k\bm a_k)|\approx n_k$; \label{le:qnk}
	\item $\bm a_k,\bm b_k$ are admissible sequences consisting of integers which lie in~$[-M,-2]\cup[2,M]$ for some constant~$M>0$; \label{le:a,b<M}
	\item $t_k+[0;\bm a_k^t]+[0;\bm b_k]\approx\tau^{-1}$, where $\bm a_k^t$ is the reverse sequence of~$\bm a_k$. \label{le:tab}
\end{enumerate}

By Lemma~\ref{l:props} \eqref{le:mf} and~\eqref{le:z-pq}, conditions~\eqref{le:qnk} and~\eqref{le:tab} imply that, for all $z\in\cld(\bm\varpi_k)$, 
\begin{equation}\label{eq:zexact}
	|z-p(\bm w_k\bm a_k)/q(\bm w_k\bm a_k)|\approx\tau|q(\bm w_k\bm a_k)|^{-2}\approx\tau n_k^{-2}\approx\psi(|q(\bm w_k\bm a_k)|),
\end{equation}
where we also use the condition $\psi(n_k)\approx\tau n_k^{-2}$. Let $z$ be the unique number in $\bigcap_{k\ge 2}\cld(\bm\varpi_k)$. By Lemma~\ref{l:badapp} and inequality~\eqref{eq:zexact}, we can prove $z\in\xct(\psi)$ under the condition~\eqref{le:a,b<M} and an additional mild restriction on~$(\bm\varpi_k)$.

We now turn to the analysis of how to make the conditions~\eqref{le:qnk}, \eqref{le:a,b<M} and~\eqref{le:tab} hold. For~\eqref{le:qnk}, we use Lemma~\ref{l:xctN}, which is an alternative version of a technique lemma due to Bugeaud and Moreira~\cite[Lemma~2]{BugMo11}. For~\eqref{le:a,b<M} and~\eqref{le:tab}, we follow an idea of Moreira~\cite[proof of Theorem~2]{Morei18}. Roughly speaking, that is, use Lemma~\ref{l:5+5} to find an integer $t_k$ and two admissible sequences $\bm\alpha_k$ and $\bm\beta_k$ consisting of integers lie in $[-M,-2]\cup[2,M]$ such that $t_k+[0;\bm\alpha_k]+[0;\bm\beta_k]\approx\tau^{-1}$. In \cite[proof of Theorem~2]{Morei18}, Moreira then take $\bm a_k=\bm\alpha_k^t$ and $\bm b_k=\bm\beta_k$ immediately, since all sequences are admissible in the setting of regular continued fractions (in his setting, $\bm\alpha_k$ and $\bm\beta_k$ are sequences consisting of positive integers). However, in the setting of HCF, the sequence $\bm\alpha_k$ is admissible does not imply that so is $\bm\alpha_k^t$. Thus, we need to find some way to get around this difficulty. The point is the observation~\eqref{eq:2-2}, which enables us to find the desired admissible sequence $\bm a_k$ after some appropriate modification of~$\bm\alpha_k^t$.

Finally, we use Lemmas~\ref{l:OMQ} and~\ref{l:CaQ} to show that there are sufficiently many choices of sequences $(\bm\varpi_k)_k$ such that one can construct a sufficiently large Cantor-like subset of~$\xct(\psi)$ from such $(\bm\varpi_k)_k$, and then obtain the desired lower bound of the Hausdorff dimension of~$\xct(\psi)$.

\subsection{Admissible sequences in real case}\label{ss:admR}

In this subsection, we discuss the admissible sequences in real case and present several related lemmas which are needed later.

Let $I_\R=I\cap\R=\Z\setminus\{0,\pm1\}$ and $I_\R^*=\bigcup_{n\ge0}I_\R^n$ with the convenience $I_\R^0=\{\varnothing\}$, where $\varnothing$ is the empty word.

\begin{lem}[see~{\cite[page 10]{Gonza18}}]\label{l:admR}
	A infinite sequence $(a_n)_{n\ge 1}\in I_\R^\infty$ is the HCF partial quotients of some $x\in \mathfrak F\cap \R$ if and only if, for all $n\ge1$,
	\begin{equation}\label{eq:admR}
		\text{either $|a_n|\ge3$ or $a_na_{n+1}>0$}.
	\end{equation}
	Consequently, a sequence $\bm u=(u_1,\dots,u_m)\in I_\R^m$ is admissible if and only if \eqref{eq:admR} holds for $1\le n<m$.
\end{lem}

\begin{lem}\label{l:adm3}
	Let $\bm u=(u_1,\dots,u_n)\in I_\R^n$. If $\bm u$ is admissible and $|u_n|\ge 3$, then $\bm u$ is full.
\end{lem}
\begin{proof}
The lemma follows from the claim below:
\begin{equation}\label{eq:Fu}
	\dmn_{\bm u}=\dmn_{u_n}\quad \text{for any admissible sequence $\bm u=(u_1,\dots,u_n)\in I_\R^n$}.
\end{equation}
Indeed, the claim together with Example~\ref{e:dmn2} gives $\dmn_{\bm u}=\dmn_{u_n}=\dmn$, since $|u_n|\ge3$. Therefore, $\bm u$ is full (see Definition~\ref{d:regfl}).

It remains to prove claim~\eqref{eq:Fu}. We proceed by induction on~$n$. For $n=1$, there is nothing to prove. Suppose now that claim~\eqref{eq:Fu} holds for $n-1$. Let $\bm u=(u_1,\dots,u_n)\in I_\R^n$ be an admissible sequence. By the induction hypothesis and Lemma~\ref{l:image} \eqref{le:uv=u'v}, we have 
\[ \dmn_{\bm u}=\dmn_{u_1,\dots,u_{n-1},u_n}=\dmn_{u_{n-1},u_n}. \]
By Lemma~\ref{l:image} \eqref{le:image}, this implies $\dmn_{\bm u}=\dmn_{u_{n-1},u_n}=\dmn_{u_n}$ if we establish $\cld(u_n)\subset\dmn_{u_{n-1}}$. To prove this, we consider three cases: $|u_{n-1}|\ge3$, $u_{n-1}=2$ or $u_{n-1}=-2$.

\paragraph{Case~1: $|u_{n-1}|\ge3$} Example~\ref{e:dmn2} gives $\dmn_{u_{n-1}}=\dmn$, and so $\cld(u_n)\subset\dmn_{u_{n-1}}$.

\paragraph{Case~2: $u_{n-1}=2$} Since $\bm u$ is admissible, we have $u_n\ge2$ by~\eqref{eq:admR}. It follows that, for all $z\in\cld(u_n)$,
\[ |z+1|\ge 1+1/u_n-|z-1/u_n|\ge 1+1/u_n-1/u_n^2>1. \]
Here, we use Lemma~\ref{l:props} \eqref{le:z-pq}, which asserts $|z-1/u_n|\le1/u_n^2$. Recall from Example~\ref{e:dmn2} that $\dmn_2=\dmn\setminus\overline B(-1,1)$, we obtain $\cld(u_n)\subset\dmn_2$. 

\paragraph{Case~3: $u_{n-1}=-2$} It is similar to the case $u_{n-1}=2$.
\end{proof}

The following lemma enables us to obtain certain admissible sequences, which plays an important role in our method of finding points in~$\xct(\psi)$.
\begin{lem}\label{l:rev}
	For $n\ge2$, let $\bm u=(u_1,\dots,u_n)\in I_\R^n$ with $u_1u_2>0$. If its reverse sequence $\bm u^t=(u_n,\dots,u_1)$ is admissible, then there exists an admissible sequence $\bm v=(v_1,\dots,v_n)\in I_\R^n$ satisfying
	\begin{enumerate}[\upshape(a)]
		\item \label{le:utw=vtw}$[0;\bm u^t\bm w]=[0;\bm v^t\bm w]$ for all $\bm w\in I^*$;
		\item \label{le:|v|<|u|} $|v_j|\le|u_j|+1$ for $1\le j\le n$;
		\item \label{le:unvn} $u_nv_n>0$.
	\end{enumerate}
\end{lem}
\begin{proof}
We prove this by induction on the number
\[ \eta_{\bm u}:=\sharp\{1\le j< n\colon\text{$|u_j|=2$ and $u_ju_{j+1}<0$} \}. \]
In the case $\eta_{\bm u}=0$, $\bm u$ is admissible (see Lemma~\ref{l:admR}). So just take $\bm v=\bm u$.

Suppose now that the conclusion holds in the case $\eta_{\bm u}=k-1$ for all $n\ge2$ and all such $\bm u\in I_\R^n$. We shall prove the case $\eta_{\bm u}=k$. The key to obtaining the desired sequence $\bm v$ is the observation that, for all $x,y\in\R$,
\begin{equation}\label{eq:2-2}
	\begin{pmatrix}
		x & 1\\
		1 & 0
	\end{pmatrix}
	\begin{pmatrix}
		2 & 1\\
		1 & 0
	\end{pmatrix}\begin{pmatrix}
		y & 1\\
		1 & 0
	\end{pmatrix}
	=
	-\begin{pmatrix}
		x+1 & 1\\
		1 & 0
	\end{pmatrix}
	\begin{pmatrix}
		-2 & 1\\
		1 & 0
	\end{pmatrix}
	\begin{pmatrix}
		y+1 & 1\\
		1 & 0
	\end{pmatrix}.
\end{equation}

Let $m$ be the smallest integer such that $|u_m|=2$ and $u_mu_{m+1}<0$. Then $m\ge2$ since $u_1u_2>0$. There are two cases: $u_m=2$ or $u_m=-2$. By symmetry, we only consider one of them, the proof for the other one is similar.

Assume that $u_m=2$. Since $\bm u^t$ is admissible, by Lemma~\ref{l:admR}, we have $u_{m-1}\ge2$ and $u_{m+1}\le -3$ (since $u_{m+1}<0$). According to~\eqref{eq:2-2}, we consider the sequence
\[ (u_1,\dots,u_{m-2},u_{m-1}+1,-2,u_{m+1}+1,u_{m+2},\dots,u_n). \]
Divide it into two subsequences
\begin{equation}\label{eq:uv'}
	\tilde{\bm v}=(u_1,\dots,u_{m-2},u_{m-1}+1),\quad
	\bm u'=(-2,u_{m+1}+1,u_{m+2},\dots,u_n).
\end{equation}

For $\tilde{\bm v}$, we have $\eta_{\tilde{\bm v}}=0$ by the definition of~$m$. Hence, $\tilde{\bm v}$ is admissible. Moreover, Lemma~\ref{l:adm3} implies that $\tilde{\bm v}$ is full since $u_{m-1}+1\ge3$. As for $\bm u'$, writing it as $(u'_1,\dots,u'_{n-m+1})$, we have $u'_1u'_2=-2(u_{m+1}+1)>0$ and the reverse sequence $(\bm u')^t$ is also admissible. In other words, $\bm u'$ satisfies all the conditions of the lemma. Furthermore, $\eta_{\bm u'}=k-1$. By the induction hypothesis, there is an admissible sequence $\bm v'=(v'_1,\dots,v'_{n-m+1})$ satisfying the conclusions~\eqref{le:utw=vtw}, \eqref{le:|v|<|u|} and~\eqref{le:unvn} corresponding to~$\bm u'$. Since $\tilde{\bm v}$ is full and $\bm v'$ is admissible, the sequence $\bm v=\tilde{\bm v}\bm v'$ is admissible. We shall show that $\bm v$ is the desired sequence.

For conclusion~\eqref{le:utw=vtw}, it suffices to show that $[0;\bm u^t\bm w]=[0;(\bm u')^t\tilde{\bm v}^t\bm w]$ for $\bm w\in I^*$ since conclusion~\eqref{le:utw=vtw} also holds for $\bm u'$ and $\bm v'$. By~\eqref{eq:Qpair}, \eqref{eq:2-2} and~\eqref{eq:uv'}, we have
\[ \begin{pmatrix}
	q(\bm u^t\bm w) & q(\bm u^t\bm w^-) \\
	p(\bm u^t\bm w) & p(\bm u^t\bm w^-)
\end{pmatrix} = - \begin{pmatrix}
	q((\bm u')^t\tilde{\bm v}^t\bm w) & q((\bm u')^t\tilde{\bm v}^t\bm w^-) \\
	p((\bm u')^t\tilde{\bm v}^t\bm w) & p((\bm u')^t\tilde{\bm v}^t\bm w^-)
\end{pmatrix}. \]
Hence, $[0;\bm u^t\bm w]=p(\bm u^t\bm w)/q(\bm u^t\bm w)=p((\bm u')^t\tilde{\bm v}^t\bm w)/q((\bm u')^t\tilde{\bm v}^t\bm w)=[0;(\bm u')^t\tilde{\bm v}^t\bm w]$.

For conclusion~\eqref{le:|v|<|u|}, recall that $\bm u'=(-2,u_{m+1}+1,u_{m+2},\dots,u_n)$ and
\[ \bm v=\tilde{\bm v}\bm v'=(u_1,\dots,u_{m-2},u_{m-1}+1,v'_1,\dots,v'_{n-m+1}). \]
Therefore, $v_j=u_j$ for $1\le j\le m-2$; $v_{m-1}=u_{m-1}+1$ and 
\[ |v_j|=|v'_{j-m+1}|\le|u'_{j-m+1}|+1=|u_j|+1 \]
for $m+2\le j\le n$. Since $u_m=2$ and $u_{m+1}\le-3$, we also have
\begin{align*}
	|v_m|&=|v'_1|\le|u'_1|+1=|-2|+1=|u_m|+1; \\
	|v_{m+1}|&=|v'_2|\le|u'_2|+1=|u_{m+1}+1|+1=|u_{m+1}|.
\end{align*}
Summarizing, we have shown $|v_j|\le|u_j|+1$ for $1\le j\le n$.

For conclusion~\eqref{le:unvn}, first note that $n\ge m+1$. If $n=m+1$, then
\[ u_nv_n=u_{m+1}v_{m+1}=(u'_2-1)v'_2>0, \]
since $|u'_2|>1$ and $u'_2v'_2=u'_{n-m+1}v'_{n-m+1}>0$ (by the induction hypothesis). If otherwise $n\ge m+2$, then $u_nv_n=u'_{n-m+1}v'_{n-m+1}>0$.
\end{proof}

The lemma below is needed in constructing Cantor-like subsets in~$\xct(\psi)$, which is an alternative version of a technique lemma due to Bugeaud and Moreira~\cite[Lemma~2]{BugMo11}. We omit its proof since the argument is similar.
\begin{lem}\label{l:xctN}
	For any $\delta>0$ and any admissible sequence $\bm w\in I_\R^*$, there exists a positive constant $\ell=\ell(\delta,\bm w)$ such that, for any positive number $N$ and any full sequence $\bm u\in I^*$ with $|q(\bm u)|<N/\ell$, we can find a full sequence $\bm v=\bm v_{\bm u}\in\bigcup_{n\ge0}\{3,4\}^n$ satisfying
	\[ |q(\bm u\bm v\bm w)|\in[(1-\delta)N,N]. \]
\end{lem}

\subsection{Exact approximation}\label{ss:exact}

The aim of this subsection is to prove Lemma~\ref{l:akbk}, which enables us to find points which satisfy the exact approximation property.

For $m\ge2$, let 
\[ \HCF_\R(m):=\bigl\{x\in[-1/2,1/2)\colon |a_n(z)|\le m\bigr\}. \]

\begin{lem}[{\cite[Theorem 41]{Brouw19}}]\label{l:5+5}
	$\Z+\HCF_\R(5)+\HCF_\R(5)=\R$.
\end{lem}

\begin{lem}\label{l:sum}
	Every real number can be written as the sum $t+\alpha+\beta$, where $t$ is an integer, $\alpha\in[0,1/2)\cap\HCF_\R(29)$ and $\beta\in\HCF_\R(29)$.
\end{lem}
\begin{proof}
	Lemma~\ref{l:5+5} says that every real number can be written as the sum $t+\alpha+\beta$, where $t$ is an integer and $\alpha,\beta\in \HCF_\R(5)$. If one of $\alpha$, $\beta$ lies in $[0,1/2)$, then we are done. Now suppose that $\alpha,\beta\in [-1/2,0)$, we claim that $\alpha+1/2,\beta+1/2\in \HCF_\R(29)$, then $(t-1)+(\alpha+1/2)+(\beta+1/2)$ is the desired sum.
	
	Assume otherwise that $\alpha+1/2\notin\HCF_\R(29)$. Let $(a_n)_n$ be the HCF partial quotients of $\alpha+1/2$, then $|a_{n}|\ge30$ for some $n$. Let $p_{n-1}=p(a_1,\dots,a_{n-1})$, $q_{n-1}=q(a_1,\dots,a_{n-1})$, then by Lemma~\ref{l:props} \eqref{le:z-pq},
	\[\left|\alpha-\frac{2p_{n-1}-q_{n-1}}{2q_{n-1}}\right|=\left|\alpha+\frac{1}{2}-\frac{p_{n-1}}{q_{n-1}}\right|< \frac{1}{(|a_n|-2)q_{n-1}^2}
	\le \frac{1}{28q_{n-1}^2}=\frac{1}{7\cdot(2q_{n-1})^2}.\]
	
	On the other hand, since $\alpha\in \HCF_\R(5)$, we infer from Theorem~\ref{l:props}~\eqref{le:z-pq} and Lemma~\ref{l:badapp} that $|\alpha-p/q|> 1/(7|q|^2)$ for any $p/q\in\Q(i)$, a contradiction! Hence $\alpha+1/2\in\HCF_\R(29)$. Similarly, one can verify that so does $\beta+1/2$.
\end{proof}

\begin{lem}\label{l:w3}
	Let $\bm w=(w_1,\dots,w_n)\in I^*$, then $\bigl|[0;\bm w]\bigr|<1/(|w_1|-1)$.
\end{lem}
\begin{proof}
	We prove this by induction on~$n$.	For $n=1$, $\bigl|[0;w_1]\bigr|=1/|w_1|<1/(|w_1|-1)$. Now suppose that this holds for all sequence in~$I^*$ of length~$n-1$ with $n\ge2$. Since
	\[ [0;\bm w]=\frac{1}{w_1+p(w_2,\dots,w_n)/q(w_2,\dots,w_n)}, \]
	it suffices to show that $|p(w_2,\dots,w_n)|<|q(w_2,\dots,w_n)|$. This follows from the induction hypothesis since
	\[ |p(w_2,\dots,w_n)/q(w_2,\dots,w_n)|=\bigl|[0;w_2,\dots,w_n]\bigr|<1/(|w_2|-1)\le1. \qedhere \]
\end{proof}

\begin{lem}\label{l:akbk}
	Given two real numbers $0<\zeta<\xi<1/3$, there exist an integer $t$ and two sequences $\bm a=(a_1,\dots,a_m)$, $\bm b=(b_1,\dots,b_n)$ in $I_\R^*$ such that
	\begin{enumerate}[\upshape(a)]
		\item \label{le:t} $3\le t<\zeta^{-1}+1$.
		\item \label{le:abfull} the sequence $\bm at\bm b$ is full;
		\item \label{le:ab<} $|a_j|\le30$ for $1\le j\le m$ and $|b_j|\le29$ for $1\le j\le n$;
		\item \label{le:atb} for all full sequence $\bm w\in I^*$ and all $z\in\cld(\bm w\bm at\bm b)$,
		\[ \zeta<|z-p(\bm w\bm a)/q(\bm w\bm a)|\cdot|q(\bm w\bm a)|^2<\xi. \]
	\end{enumerate}
\end{lem}
\begin{proof}
By Lemma~\ref{l:sum}, we can find $t\in\Z$, $\alpha\in[0,1/2)\cap\HCF_\R(29)$ and $\beta\in\HCF_\R(29)$ such that $\zeta^{-1}>t+\alpha+\beta>\xi^{-1}>3$. This implies conclusion~\eqref{le:t} since $t\in\Z$ and $|\alpha|,|\beta|<1/2$.

Let $\bm\alpha=(\alpha_1,\dots,\alpha_{m-1})$ and $\bm\beta=(\beta_1,\dots,\beta_{n-1})$ be prefixes of HCF partial quotients of~$\alpha$ and~$\beta$, respectively, such that
\begin{equation}\label{eq:a'b'}
	\zeta^{-1}>\bigl|t+[0;\bm\alpha\bm w]+z\bigr|>\xi^{-1}
\end{equation}
for all $\bm w=(w_1,\dots,w_k)\in I^*$ with $|w_1|\ge3$ and all $z\in\cld(\bm\beta)$. Such two prefixes do exist, since
\begin{gather*}
	[0;\bm\alpha\bm w] = \frac{p(\bm\alpha)q(\bm w)+p(\bm\alpha^-)p(\bm w)}{q(\bm\alpha)q(\bm w)+q(\bm\alpha^-)p(\bm w)} =\frac{p(\bm\alpha)+[0;\bm w]\cdot p(\bm\alpha^-)}{q(\bm\alpha)+[0;\bm w]\cdot q(\bm\alpha^-)} \to\alpha,\\
	\beta, z\in\cld(\bm\beta)\quad \text{and}\quad |\cld(\bm\beta)|\to0,
\end{gather*}
as the lengths of the prefixes tend to infinity. Here we use~\eqref{eq:Qpair} and the fact $\bigl|[0;\bm w]\bigr|<1/2$ (by Lemma~\ref{l:w3}) in the first line.

Set $b_n=3$ if $\beta_{n-1}>0$ or $b_n=-3$ if $\beta_{n-1}<0$. Let $\bm b=(b_1,\dots,b_n)=\bm\beta b_n$. Since $\bm\beta$ is admissible, $b_{n-1}b_n=\beta_{n-1}b_n>0$ and $|b_n|=3$, it follows from Lemmas~\ref{l:admR} and~\ref{l:adm3} that the sequence $\bm b$ is full.

Set $a'_m=3$ if $\alpha_{m-1}>0$ or $a'_m=-3$ if $\alpha_{m-1}<0$. Similarly, the sequence $\bm\alpha a'_m$ is full. Now consider the reverse sequence $\tilde{\bm a}=(\tilde a_1,\dots,\tilde a_m)=(\bm\alpha a'_m)^t$, which satisfies the conditions of Lemma~\ref{l:rev} since $\tilde{a}_1\tilde{a}_2=a'_m\alpha_{m-1}>0$. Therefore, there is an admissible sequence $\bm a=(a_1,\dots,a_m)$ satisfying the conclusions~\eqref{le:utw=vtw}, \eqref{le:|v|<|u|} and~\eqref{le:unvn} of Lemma~\ref{l:rev} corresponding to~$\tilde{\bm a}$. In particular,
\begin{equation}\label{eq:aa}
	[0;\bm a^t\bm w]=[0;\tilde{\bm a}^t\bm w]=[0;\bm\alpha a'_m\bm w]\quad\text{for all $\bm w\in I^*$}.
\end{equation}
It remains to show that such $t$, $\bm a$ and $\bm b$ are desired.

For conclusion~\eqref{le:abfull}, since $\bm b$ is full, it suffices to prove $\bm at$ is also full. Recall that $\bm a$ is admissible and $t\ge3$, by Lemmas~\ref{l:admR} and~\ref{l:adm3}, we only need to verify that $a_m>0$. Since $\alpha\in[0,1/2)$, we have $\alpha_1>0$. Moreover, Lemma~\ref{l:rev} \eqref{le:unvn} implies that $a_m\alpha_1=a_m\tilde{a}_m>0$. Hence, $a_m>0$ and the proof of~\eqref{le:abfull} is completed.

Conclusion~\eqref{le:ab<} follows from the facts $\alpha,\beta\in\HCF_\R(29)$, $|a'_m|=|b_n|=3$ and conclusion~\eqref{le:|v|<|u|} of Lemma~\ref{l:rev}.

For conclusion~\eqref{le:atb}, denote by $\tilde{n}$ the length of $\bm w\bm a$, then apply Lemma~\ref{l:props}~\eqref{le:z-pq} and~\eqref{le:mf} to obtain that, for $z\in\cld(\bm w\bm at\bm b)$,
\[ |z-p(\bm w\bm a)/q(\bm w\bm a)|\cdot|q(\bm w\bm a)|^2= \bigl|t+z'+[0;\bm a^t\bm w^t]\bigr|^{-1}, \]
where $z'=T^{\tilde{n}+1}(z)\in\cld(\bm b)\subset\cld(\bm\beta)$. Moreover, since $|a'_m|=3$, combine~\eqref{eq:a'b'} and~\eqref{eq:aa} to get
\[ \bigl|t+z'+[0;\bm a^t\bm w^t]\bigr|^{-1}=\bigl|t+z'+[0;\bm\alpha a'_m\bm w^t]\bigr|^{-1}\in(\zeta,\xi). \qedhere \]
\end{proof}

\subsection{An auxiliary family of full sequences} \label{ss:para2}

In this subsection, we introduce an auxiliary families of full sequences, which enables us to construct the desired Cantor-like subsets of~$\xct(\psi)$.

Fix a non-increasing function $\psi\colon(0,\infty)\to(0,\infty)$ with
\begin{equation}\label{eq:tau0}
	\tau=\limsup_{x\to+\infty}x^2\psi(x)\le\tau_0:=\min\bigl(1/(12\cst_1+2),1/32\bigr),
\end{equation}
where $\cst_1>1$ is the constant given by Lemma~\ref{l:OMQ}. Let \[ M=1/\tau-2\ge\max(12\cst_1,30). \]
For $k\ge2$, apply Lemma~\ref{l:akbk} with $\zeta=(1-1/(2k))\tau$ and $\xi=(1-1/(3k))\tau$ to obtain $t_k$, $\bm a_k$ and~$\bm b_k$ such that
\begin{gather}
	t_k<2/\tau; \label{eq:t<2tau} \\
	\text{$\bm a_kt_k\bm b_k$ is full and $\bm a_k,\bm b_k\in I_\R^*\cap I_M^*$}; \label{eq:akbk1}\\
	(1-1/(2k))\tau<|z-p(\bm w\bm a_k)/q(\bm w\bm a_k)|\cdot|q(\bm w\bm a_k)|^2<(1-1/(3k))\tau \label{eq:akbk2}
\end{gather}
for all full sequence $\bm w\in I^*$ and all $z\in\cld(\bm w\bm a_kt_k\bm b_k)$.

Pick a rapidly increasing sequence~$(n_k)_{k\ge1}$ of positive numbers with $n_1=1$ such that, for $k\ge2$,
\begin{gather}
	(1-1/(3k))\tau<x^2\psi(x)<(1+1/(2k))\tau \quad\text{for $x\in[(1-1/(9k))n_k,n_k]$}; \label{eq:nkpsi2}\\
	n_k^{1/k}\ge\max\bigl(3(M+1)n_{k-1}^{1+1/(k-1)}\ell_k,3|q(t_k\bm b_k)|\bigr); \label{eq:nk>k-1} \\
	Q_k:=n_k^{1-1/k}\ge(M+1)^{5M}. \label{eq:Qk2}
\end{gather}
Here $\ell_k=\ell(1/(9k),\bm a_k)$ is given by Lemma~\ref{l:xctN}. To see that we can make \eqref{eq:nkpsi2} true, recall that $\tau=\limsup_{x\to+\infty}x^2\psi(x)$, so the right inequality of~\eqref{eq:nkpsi2} holds for sufficiently large~$x$; for the right inequality, note that $(1-1/(3k))/(1-1/(9k))^2<1$, hence we can choose $n_k$ such that $n_k^2\psi(n_k)>\tau\cdot(1-1/(3k))/(1-1/(9k))^2$, then by the fact that $\psi$ is non-increasing, we conclude
\[ x^2\psi(x)\ge(1-1/(9k))^2n_k^2\psi(n_k)>(1-1/(3k))\tau \]
for $x\in[(1-1/(9k))n_k,n_k]$.

We shall define the families $(\Lambda_k)_{k\ge1}$ of full sequences by induction on~$k$ such that
\[ |q(\bm a)|\le n_k^{1+1/k}\quad\text{for $\bm a\in\Lambda_k$}. \]

For $k=1$, let $\Lambda_1=\{\varnothing\}$. Clearly, it meets the requirements.

For $k\ge2$, suppose now that $\Lambda_{k-1}$ has been defined. Let $\Gamma_M(Q_k)$ be as in Lemma~\ref{l:OMQ} with $Q_k=n_k^{1-1/k}$ given by~\eqref{eq:Qk2}. Given $\hat{\bm a}\in\Lambda_{k-1}$ and $\bm u\in\Gamma_M(Q_k)$, then by Lemma~\ref{l:props} \eqref{le:|q-|}, $|q(\bm u)|\le(M+1)Q_k$. This together with Lemma~\ref{l:props} \eqref{le:|qq|}, \eqref{eq:nk>k-1} and the induction hypothesis that $|q(\hat{\bm a})|\le n_{k-1}^{1+1/(k-1)}$ gives
\[ |q(\hat{\bm a}\bm u)|\le 3|q(\hat{\bm a})q(\bm u)|\le 3n_{k-1}^{1+1/(k-1)}\cdot (M+1)Q_k\le n_k/\ell_k. \]
Note also that $\hat{\bm a}$ is full by induction hypothesis,  so we can apply Lemma~\ref{l:xctN} with $\delta=1/(9k)$ and $\bm w=\bm a_k$ to get a full sequence $\bm v=\bm v_{\hat{\bm a}\bm u}$ satisfying
\begin{equation}\label{eq:xctnk}
	|q(\hat{\bm a}\bm u\bm v\bm a_k)|\in[(1-1/(9k))n_k,n_k].
\end{equation}
Let $\Lambda_k$ be the family consisting of sequences of the form $\hat{\bm a}\bm u\bm v\bm a_kt_k\bm b_k$, i.e.,
\begin{equation}\label{eq:Lamk}
	\Lambda_k=\bigl\{\hat{\bm a}\bm u\bm v\bm a_kt_k\bm b_k\colon\hat{\bm a}\in\Lambda_{k-1},\bm u\in\Gamma_M(Q_k),\bm v=\bm v_{\hat{\bm a}\bm u}\ \text{by Lemma~\ref{l:xctN}}\bigr\}.
\end{equation}

Now let $\bm a=\hat{\bm a}\bm u\bm v\bm a_kt_k\bm b_k\in\Lambda_k$, we need to prove $\bm a$ is full and $|q(\bm a)|\le n_k^{1+1/k}$. From the induction hypothesis, definition of~$\Gamma_M(Q_k)$ (see Lemma~\ref{l:OMQ}), Lemma~\ref{l:xctN} and~\eqref{eq:akbk1}, the sequences $\hat{\bm a}$, $\bm u$, $\bm v$ and $\bm a_kt_k\bm b_k$ are all full. Hence, so is $\bm a=\hat{\bm a}\bm u\bm v\bm a_kt_k\bm b_k$ by Lemma~\ref{l:image} \eqref{le:image}. As for the upper bound for $|q(\bm a)|$, by Lemma~\ref{l:props} \eqref{le:|qq|}, \eqref{eq:nk>k-1} and~\eqref{eq:xctnk},
\[ |q(\bm a)|=|q(\hat{\bm a}\bm u\bm v\bm a_kt_k\bm b_k)|\le 3|q(\hat{\bm a}\bm u\bm v\bm a_k)q(t_k\bm b_k)|\le n_k^{1+1/k}. \]
Thus, we have finished the definition of~$\Lambda_k$.

We summarize some properties of~$\Lambda_k$ in the lemma below.
\begin{lem}\label{l:Lamk}
	For $k\ge2$, the following statements hold.
	\begin{enumerate}[\upshape(a)]
		\item \label{le:Lamfull} All the sequence in $\Lambda_k$ is full.
		\item \label{le:Lamq<} For $\bm a\in\Lambda_k$, $|q(\bm a)|\le n_k^{1+1/k}$.
		\item \label{le:<2tau} $\Lambda_k\subset I_{2/\tau}^*$.
		\item \label{le:NLamk} $\sharp\Lambda_k=\sharp\Lambda_{k-1}\cdot\sharp\Gamma_M(Q_k)\ge n_k^{2d_\tau(1-1/k)}$, where $d_\tau=2-\tau/(1-2\tau)$.
		\item \label{le:z-pqtau} For all $\bm a=\hat{\bm a}\bm u\bm v\bm a_kt_k\bm b_k\in\Lambda_k$ and all $z\in\cld(\bm a)$, let $\tilde{\bm a}=\hat{\bm a}\bm u\bm v\bm a_k$, then
		\[ (1-1/k)\psi(|q(\tilde{\bm a})|)<|z-p(\tilde{\bm a})/q(\tilde{\bm a})|<\psi(|q(\tilde{\bm a})|). \]
	\end{enumerate}
\end{lem}
\begin{proof}
	(a) and (b) have been proved in the process of definition of~$\Lambda_k$.
	
	(c) This follows from the definition of~$\Lambda_k$ (see~\eqref{eq:Lamk}), \eqref{eq:t<2tau}, \eqref{eq:akbk1} and the fact $2/\tau>M\ge30$.

	(d) By the definition of~$\Lambda_k$ (see~\eqref{eq:Lamk}), since $\bm a_k$, $t_k$, $\bm b_k$ are all fixed and $\bm v$ is also fixed once $\hat{\bm a}$ and $\bm u$ are chosen, we have $\sharp\Lambda_k= \sharp\Lambda_{k-1}\cdot\sharp\Gamma_M(Q_k)$. Then use~\eqref{eq:Qk2} and apply Lemma~\ref{l:OMQ} to obtain
	\[ \sharp\Lambda_k\ge \sharp\Gamma_M(Q_k)\ge Q_k^{4-2/M}\ge n_k^{2d_\tau(1-1/k)}. \]
	Here we use the identity $M=1/\tau-2$.
	
	(e) By~\eqref{eq:nkpsi2} and~\eqref{eq:xctnk}, we have
	\[ (1-1/(3k))\tau<\psi(|q(\tilde{\bm a})|)|q(\tilde{\bm a})|^2<(1+1/(2k))\tau. \]
	This together with~\eqref{eq:akbk2} gives the desired inequalities.
\end{proof}

\subsection{Lower bound for Hausdorff dimension}\label{ss:lb2}

Define 
\begin{equation}\label{eq:Etau}
	E_\tau=\bigcap_{k\ge1}\bigcup_{\bm a\in\Lambda_k}\cld(\bm a).
\end{equation}
By the same argument after~\eqref{eq:Ecapa}, one can verify that $E_\tau\cap\cld(\bm a)\ne\emptyset$ for all $\bm a\in\Lambda_k$ with $k\ge1$. Using this fact, we can define a probability measure~$\mu$ on~$E_\tau$ such that
\begin{equation}\label{eq:mua2}
	\mu(E_\tau\cap\cld(\bm a))=(\sharp\Lambda_k)^{-1}\quad\text{for all $\bm a\in\Lambda_k$ and all $k\ge1$}.
\end{equation}
To see that $\mu$ is well-defined, pick $\hat{\bm a}\in\Lambda_{k-1}$, let $\Lambda_k^{\hat{\bm a}}=\{\bm a\in\Lambda_k\colon\hat{\bm a}\prec\bm a\}$. By~\eqref{eq:Lamk}, we have $\sharp\Lambda_k^{\hat{\bm a}}=\sharp\Gamma_M(Q_k)$. Combining with Lemma~\ref{l:Lamk} \eqref{le:NLamk},
\[ \mu(E_\tau\cap\cld(\hat{\bm a}))=(\sharp\Lambda_{k-1})^{-1}= \sharp\Gamma_M(Q_k)\cdot(\sharp\Lambda_k)^{-1}=\sum_{\bm a\in\Lambda_k^{\hat{\bm a}}}\mu(E_\tau\cap\cld(\bm a)). \]
Thus, such a measure~$\mu$ does exist by Carath\'eodory's extension theorem.

\begin{lem}\label{l:mu<}
	Let $d_\tau=2-\tau/(1-2\tau)$. Then for all $z\in E_\tau$,
	\[ \liminf_{r\to0}\frac{\log\mu(B(z,r))}{\log r}=d_\tau. \]
\end{lem}
\begin{proof}
Given $z\in E_\tau$ and $0<r<1$, we want to find all the cylinders that intersect $B(z,r)$ and of diameter close to~$r$. To this end, we consider the family $\Lambda(z,r)$ consisting of sequence~$\bm e$ such that 
\begin{equation}\label{eq:ezr}
	\cld(\bm e)\cap E_\tau\cap B(z,r)\ne\emptyset \quad\text{and}\quad |q(\bm e)|^{-2}\le r<|q(\bm e^-)|^{-2}.
\end{equation}
To bound $\mu(B(z,r))$ from above, we need to find the upper bound for $\sharp\Lambda(z,r)$ and $\max_{\bm e\in\Lambda(z,r)}\mu(\cld(\bm e))$.

We begin with $\sharp\Lambda(z,r)$. Pick $\bm e\in\Lambda(z,r)$, since $\cld(\bm e)\cap E_\tau\ne\emptyset$, $\bm e$ must be a prefix of some $\bm a\in\Lambda_j$ for $j$ large enough. Hence, we have $\bm e\in I_{2/\tau}^*$ by Lemma~\ref{l:Lamk} \eqref{le:<2tau}. This together with Lemma~\ref{l:props} \eqref{le:|q-|} implies $|q(\bm e)|<(2/\tau+1)|q(\bm e^-)|<3|q(\bm e^-)|/\tau$. By Lemma~\ref{l:diamea}, we obtain
\[ \lm(\cld(\bm e))\ge\cst_0\pi|q(\bm e)|^{-4}\ge\cst_0\pi(3|q(\bm e^-)|/\tau)^{-4}\ge \cst_0\pi\tau^4r^2/81. \]
Moreover, Lemma~\ref{l:diamea} also gives $|\cld(\bm e)|\le2|q(\bm e)|^{-2}\le 2r$. Thus, $\cld(\bm e)\subset B(z,3r)$ for all $\bm e\in\Lambda(z,r)$. Consequently,
\begin{equation}\label{eq:Nzr<}
	\sharp\Lambda(z,r)\le\frac{\lm(B(z,3r))}{\min_{\bm e\in\Lambda(z,r)}\lm(\cld(\bm e))}\le\frac{9\pi r^2}{\cst_0\pi\tau^4r^2/81}=729\cst_0^{-1}\tau^{-4}.
\end{equation}

We now turn to estimate $\mu(\cld(\bm e))$. Fix $k\ge1$ such that 
\begin{equation}\label{eq:rk}
	n_{k+1}^{-2(1+1/(k+1))}\le r<n_k^{-2(1+1/k)}.
\end{equation}
For $\bm e\in\Lambda(z,r)$, we claim that $\bm e=\bm a\bm w$ for some $\bm a\in\Lambda_k$. In fact, it follows from \eqref{eq:ezr}, \eqref{eq:rk} and Lemma~\ref{l:Lamk} \eqref{le:Lamq<} that, for all $\bm a\in\Lambda_k$,
\begin{equation}\label{eq:qank}
	|q(\bm e)|\ge r^{-1/2}>n_k^{1+1/k}\ge|q(\bm a)|.
\end{equation}
Since $\cld(\bm e)\cap E_\tau\ne\emptyset$, we have $\cld(\bm e)\subset\cld(\bm a)$ for some $\bm a\in\Lambda_k$, and so $\bm e=\bm a\bm w$.

We divide the remaining proof into two cases.

\paragraph{Case~1: $\bm w$ is not a prefix of some $\bm u\in\Gamma_M(Q_{k+1})$} By the definition of~$\Lambda_k$ (see~\eqref{eq:Lamk}), this means that there exists only one sequence $\bm a'\in\Lambda_{k+1}$ such that $\cld(\bm a')\cap\cld(\bm e)\ne\emptyset$. Combining with Lemma~\ref{l:Lamk} \eqref{le:NLamk} and~\eqref{eq:rk} gives
\begin{equation}\label{eq:mue<}
	\mu(\cld(\bm e))\le\mu(\cld(\bm a'))=(\sharp\Lambda_{k+1})^{-1}\le n_{k+1}^{-2d_\tau k/(k+1)}\le r^{d_\tau k/(k+2)}.
\end{equation}

\paragraph{Case~2: $\bm w$ is a prefix of some $\bm u\in\Gamma_M(Q_{k+1})$} By the definition of~$\Lambda_k$ (see~\eqref{eq:Lamk}),
\[ \sharp\{\cld(\bm a')\subset\cld(\bm e)\colon\bm a'\in\Lambda_{k+1}\}=\sharp\{\bm b\colon\bm w\bm b\in\Gamma_M(Q_{k+1})\}=\sharp\Gamma_M^{\bm w}(Q_{k+1}). \]
Combining with Lemma~\ref{l:CaQ} and Lemma~\ref{l:Lamk} \eqref{le:NLamk} gives
\[ \mu(\cld(\bm e))=\frac{\sharp\Gamma_M^{\bm w}(Q_{k+1})}{\sharp\Lambda_{k+1}}=\frac{\sharp\Gamma_M^{\bm w}(Q_{k+1})}{\sharp\Gamma_M(Q_{k+1})\cdot\sharp\Lambda_k}\le \frac{(M+1)^{24M}}{|q(\bm w)|^{4-2/M}\cdot n_k^{2d_\tau(1-1/k)}}. \]
Note that $4-2/M=2d_\tau$. By~\eqref{eq:qank}, $n_k^{2d_\tau(1-1/k)}\ge|q(\bm  a)|^{2d_\tau(k-1)/(k+1)}$. Hence,
\begin{equation}\label{eq:mue<r}
	\mu(\cld(\bm e))\le\frac{(M+1)^{24M}}{|q(\bm a)q(\bm w)|^{2d_\tau(k-1)/(k+1)}}\le\frac{81(M+1)^{24M}}{|q(\bm e)|^{2d_\tau(k-1)/(k+1)}}\le c\cdot r^{d_\tau(k-1)/(k+1)}
\end{equation}
with $c=81(M+1)^{24M}$. Here we use Lemma~\ref{l:props} \eqref{le:|qq|} and~\eqref{eq:qank}.

By~\eqref{eq:rk}, we have $k\to\infty$ as $r\to0$. Therefore, the lemma follows from \eqref{eq:Nzr<}, \eqref{eq:mue<} and~\eqref{eq:mue<r}.
\end{proof}

\begin{lem}\label{l:Etau<}
	$E_\tau\subset\xct(\psi)$.
\end{lem}
\begin{proof}
	Let $z\in E$. To prove $z\in\xct(\psi)$, by Lemma~\ref{l:badapp} and~\eqref{eq:tau0}, we only need to consider $|z-p/q|$ for $p/q\in\Q(i)$ is a convergent of~$z$.
	
	Let $\bm w$ be a prefix of the HCF partial quotients of~$z$, then $z\in\cld(\bm wb)$ for some $b\in I$. By the definition of~$E$ (see~\eqref{eq:Etau} and~\eqref{eq:Lamk}) and~\eqref{eq:akbk1}, we know that either $|b|\le M=1/\tau-2$ or, for some $k\ge2$,
	\begin{equation}\label{eq:wexact}
		\bm w=\hat{\bm a}\bm u\bm v\bm a_k
	\end{equation}
	with $\hat{\bm a}\in\Lambda_{k-1}$, $\bm u\in\Gamma_M(Q_k)$ and $\bm v=\bm v_{\hat{\bm a}\bm u}$ given by Lemma~\ref{l:xctN}.
	
	\paragraph{Case~1: $|b|\le1/\tau-2$} Denote the length of~$\bm w$ by~$n$, then Lemma~\ref{l:props} \eqref{le:z-pq} gives
	\[ |z-p(\bm w)/q(\bm w)|=|q(\bm w)^2(b+T^{n+1}(z)+q(\bm w^-)/q(\bm w))|^{-1}. \]
	Since $T^{n+1}(z)\subset\dmn\subset\overline B(0,\sqrt2/2)$ and $|q(\bm w^-)|<|q(\bm w)|$, we have
	\begin{multline*}
		|z-p(\bm w)/q(\bm w)|>(1/\tau-2+\sqrt2/2+1)^{-1}|q(\bm w)|^{-2}\\
		>(1/\tau-1/4)^{-1}|q(\bm w)|^{-2} >(1+\tau/4)\cdot\tau|q(\bm w)|^{-2}\ge(1+\tau/5)\psi(|q(\bm w)|)
	\end{multline*}
	for $|q(\bm w)|$ sufficiently large. Here we use~\eqref{eq:tau0} in the last inequality.
	\paragraph{Case~2: \eqref{eq:wexact} holds for some $k\ge2$} For this case, Lemma~\ref{l:Lamk} \eqref{le:z-pqtau} says that
	\[ (1-1/k)\psi(|q(\bm w)|)<|z-p(\bm w)/q(\bm w)|<\psi(|q(\bm w)|). \]
	
	Summarizing above two cases, we conclude $z\in\xct(\psi)$.
\end{proof}

\begin{proof}[Proof of Theorem~\ref{t:Ox2}]
	By Lemmas~\ref{l:mu<} and~\ref{l:Etau<}, we use the mass distribution principle (\cite[Proposition~2.3]{Falco97}) to get
	\[ \hdim\xct(\psi)\ge\hdim E_\tau\ge 2-\tau/(1-2\tau). \qedhere \]
\end{proof}

\end{document}